\documentclass[12pt]{article}
\usepackage{amsmath,amsthm,amssymb}
\usepackage[hang]{caption2}
\usepackage{indentfirst}
\usepackage[section]{placeins}
\numberwithin{equation}{section}     
\usepackage{graphics}
\usepackage{pdfsync}
\usepackage{hyperref}

\def\accentsfrancais{applemac}
  \RequirePackage[\accentsfrancais]{inputenc}

\newtheorem{thm}{Theorem}
\newtheorem{theorem}[thm]{Theorem}
\numberwithin{thm}{section}
\newtheorem{lemma}[thm]{Lemma}

\newtheorem{prop}[thm]{Proposition}

\newtheorem{remark}[thm]{Remark}

\newtheorem{definition}[thm]{Definition}

\def\ds{\displaystyle}
\def\emptyset{/\kern-.51em o}
\def\eq{\mathop{\vrule height2,6pt depth-2,3pt
         width -1pt\kern 0pt =}}

\let\norbali\normalbaselines
\def\anorbali{\norbali\advance\lineskip\jot
\advance\baselineskip\jot\advance\lineskiplimit\jot}
\def\ouvre{\let\normalbaselines\anorbali}
\def\D{{\mathbb{D}}}
\def\R{{\mathbb{R}}}

\def\V{{\mathbb{V}}}
\def\W{{\mathbb{W}}}

\def\N{\rm \hbox{I\kern-.2em\hbox{N}}}
\def\Z{\rm \hbox{Z\kern-.3em\hbox{Z}}}

\def\Ga{{\bf a}}
\def\Gb{{\bf b}}
\def\Gc{{\bf c}}

\def\Ge{{\bf e}}

\def\Gn{{\bf n}}

\def\Gt{{\bf t}}

\def\GA{{\bf A}}

\def\GD{{\bf D}}
\def\GE{{\bf E}}

\def\GI{{\bf I}}

\def\GQ{{\bf Q}}
\def\GR{{\bf R}}
\def\GS{{\bf S}}

\def\GV{{\bf V}}

\def\eps{\varepsilon}

\begin{document}
\title{ASYMPTOTIC BEHAVIOR OF STRUCTURES MADE OF  STRAIGHT RODS.}
\author{D. Blanchard$^{1}$, G. Griso$^{2}$}
\date{}
\maketitle
 
{\footnotesize
\begin{center}

$^{1}$  Universit\'e de Rouen, UMR 6085,  76801   Saint Etienne du Rouvray Cedex, France, \\ E-mail: dominique.blanchard@univ-rouen.fr

$^{2}$ Laboratoire J.-L. Lions--CNRS, Bo\^\i te courrier 187, Universit\'e  Pierre et
Marie Curie,\\ 4~place Jussieu, 75005 Paris, France, \; Email: griso@ann.jussieu.fr\\

\end{center} }

\begin{abstract}
 This paper  is devoted  to describe the deformations and the elastic energy  for structures made of straight
  rods of thickness $2\delta$   when $\delta$ tends to 0. This analysis relies on the decomposition of the large deformation of a single rod introduced in \cite{BG4} and on the extension of this technique to a multi-structure. We characterize  the asymptotic behavior of the infimum of the total elastic energy as the minimum of a limit functional  for an energy of order $\delta^\beta$ ($2<\beta\le 4$).
  
\end{abstract}
 


\section{ Introduction} 

This paper concerns the modeling of a structure ${\cal S}_\delta$ made of elastic straight rods  whose cross sections are discs of  radius $\delta$. The centerlines of the rods form
the skeleton  structure ${\cal S}$. We introduce a notion of elementary deformations of this structure based on the decomposition of a large deformation of a rod introduced in \cite {BG2}. A special care is devoted in the junctions of the rods  where these elementary deformations are translation-rotations.  An elementary deformation is characterized by two fields defined on the skeleton ${\cal S}$. The first one ${\cal V}$ stands for the centerlines deformation while the second one $\GR$ represents the rotations of the cross sections. For linearized deformations in plates or rods  structures such decompositions have been considered in \cite{GSP} and \cite{GSR}.

Then to an arbitrary deformation $v$ of ${\cal S}_\delta$, we associate an elementary deformation $v_e$  such that the residual part $\overline{v}=v-v_e$ is controlled by   $||\hbox{dist}(\nabla_x v,SO(3))||_{L^2({\cal S}_{\delta})}$ and $\delta$. In order to construct the deformation $v_e$ we first apply the rigidity theorem of \cite{FJM} -in the form given in \cite{BG2}- in a neighborhood  of each junction  to obtain  a constant translation-rotation in each junction. Then we match the decomposition derived in \cite{BG2} in each rod with this constant  translation-rotation. Doing such, we obtain estimates on ${\cal V}$ and $\GR$  for the whole structure ${\cal S}_\delta$ (see Theorem 3.3). Upon the assumption that the structure is fixed on some extremities, these estimates allow us to established a nonlinear Korn's type inequality for the admissible deformations.  Moreover we are in a position to analyze the asymptotic behavior of the Green-St Venant's strain tensor $E(v_\delta)={1/2}\big((\nabla v_\delta)^T\nabla v_\delta-\GI_3\big)$ for a sequence $v_\delta$ of admissible deformations such that $||\hbox{dist}(\nabla_x v_\delta,SO(3))||_{L^2({\cal S}_{\delta})}= O(\delta^K)$ for $1< K \le2$.

We then consider an elastic structure, whose  energy is denoted $W$, submitted to applied body forces $f_{\kappa,\delta}$ whose order with respect to  the parameter $\delta$ is related a constant $\kappa\ge 1$ (see below the order of $f_{\kappa,\delta}$). The total energy is given by $\ds J_{\kappa,\delta}(v)=\int_{{\cal S}_\delta} W(E( v))-\int_{{\cal S}_\delta} f_{\kappa,\delta}\cdot (v-I_d)$ if $\det (\nabla v)>0$ where $I_d$ is the identity map. We adopt a few usual assumptions on $W$ (see \eqref{HypW}). We set
$$ m_{\kappa,\delta}=\inf_{v\in \D_\delta}J_{\kappa,\delta}(v),$$ where $\D_\delta$ is the set of admissible deformations.

 We assume that the order of  $f_{\kappa,\delta}$ is equal to $\delta^{2\kappa-2}$ if $1\le\kappa\le2$ (or $\delta^{\kappa}$ if $\kappa\ge 2$) outside the junctions and that $f_{\kappa,\delta}$ is constant in each junction with order  $\delta^{2\kappa-3}$ if $1\le\kappa\le2$ (or $\delta^{\kappa-1}$ if $\kappa\ge 2$).  Then using the Korn's type inequality  mentioned above allows us to show that the  order of $m_{\kappa,\delta}$ is   $\delta^{2\kappa}$.    The aim of this paper is  to prove that the sequence $\ds {m_{\kappa,\delta}\over \delta^{2\kappa}}$ converges (when $\delta$ tends to 0) and to 
characterize its limit $m_\kappa$ as the minimum of a functional defined on ${\cal S}$ for $1<\kappa\le2$ (we will analyze the case $ \kappa >2 $ in a forthcoming paper). Indeed the derivation of this functional relies  on the asymptotic behavior of the Green-St Venant's strain tensor for minimizing sequences. The limit centerline deformation ${\cal V}$ is linked to the limit $\GR$ of the rotation fields via its derivatives and the  centerline directions. For  $1<\kappa<2$, the limit energy depends linearly on the two fields $({\cal V},\GR)$ where in particular $\GR$ takes its value in the convex hull of $SO(3)$. In the case $\kappa=2$, the limit of the rotation $\GR$  takes its value in $SO(3)$ and the functional is quadratic with respect to its derivatives.

\vskip 1mm
As  general references for the theory of elasticity we refer to \cite{Ant},  \cite{C1}. The theory of rods is developed  in e.g. \cite{Ant1} and \cite{Trab}. In the framework of linear elasticity, we refer to \cite{CDN} for the junction of a three dimensional domain and a two dimensional one, to \cite{DJR} for the junction of two rods and to \cite{DJP} for the junction of two plates. The junction between a rod and a plate in the linear case  is investigated in \cite{Gru1} and \cite{Murat} and in nonlinear elasticity in \cite{Gru2}. The decomposition of the displacements in thin structures has been introduced in \cite{GROD} and in \cite{GSP} and then used  in \cite{BGG1}, \cite{BGG2} and \cite{BG1} for the homogenization of the junction of rods and a plate. 
\vskip 1mm
This paper is organized as follows.  In Section 2 we specify the geometry of the structure.  Section 3 is  devoted to introduce the elementary deformation associated to a deformation and to establish  estimates. A Korn's type inequality for the structure is derived in Section 4. In Section 5  the usual  rescaling of each rod is recalled.  The limit behavior of the Green-St Venant's strain tensor is analyzed in Section 6 for sequences $(v_\delta)$ such that $||\hbox{dist}(\nabla_x v_\delta,SO(3))||_{L^2({\cal S}_{\delta})}= O(\delta^K)$ for $1< K \le2$. The assumptions on the elastic energy are introduced in Section 7 together with the scaling on the applied forces. The characterization of $m_\kappa$ is performed in Section 8 for $1<\kappa<2$ and in Section 9 for $\kappa=2$. At the end, an appendix contains a few technical results. The results of the present  paper have been announced in \cite{BGNOTE}.  
\section{Geometry and notations.} 

\subsection{The rod structure.}
 The Euclidian space $\R^3$ is related to the frame  $(O;\Ge_1,\Ge_2,\Ge_3)$. We denote   by $\|\cdot\|_{2}$ the euclidian norm   and by $\cdot$ the scalar product in $\R^3$.  

For an integer $N\ge1$ and  $i\in\{1,\ldots,N\}$, let  $\gamma_i$  be a segment  parametrized by $s_i$, with direction the unit vector $\Gt_i$,  origin the point $P_i\in \R^3$ and length $L_i$.  We have $\gamma_i=\varphi_i([0,L_i])$ where 
$$\varphi_i(s_i)=P_i+s_i\Gt_i,\qquad s_i\in\R.$$
So, the running point of   $\gamma_i$ is $\varphi_i(s_i)$, $0\le s_i\le L_i$. 
The  extremities of the segments make up a set denoted $\Gamma$. 

For any $i\in\{1,\ldots,N\}$, we choose a unit vector $\Gn_i$ normal to $\Gt_i$ and  we set
\begin{equation*}
\Gb_i=\Gt_i\land \Gn_i.
\end{equation*} 
 
 The structure-skeleton ${\cal S}$ is the set $\ds\bigcup_{i=1}^N\gamma_i$. The common points to two segments are called knots and the set of knots is denoted ${\cal K}$. For any $\gamma_i$, $i\in\{1,\ldots,N\}$,  we denote by $a^j_i$ the arc-length of the knots belonging to $\gamma_i$,
 $$0\le a^1_i <a_i^2< \ldots < a_i^{K_i}\le L_i.$$  
  Among all the knots,  $\Gamma_{\cal K}$ is the set of those which are extremities of all segments containing them.
 
\noindent{\bf  Geometrical hypothesis. }{\it We assume the following hypotheses on ${\cal S}$:

$\bullet$ ${\cal S}$ is connected,

$\bullet$  $\forall (i,j)\in\{1,\ldots,N\}^2$
$$i\not=j\qquad \Longrightarrow\qquad \gamma_i\cap\gamma_j=\emptyset \quad \hbox{or}\quad \gamma_i\cap\gamma_j\hbox{ is reduced to one knot.}$$}

We denote by $\omega$ the unit disc of center the origin and we set  $\omega_\delta= \delta \omega$ for   $\delta>0$. The reference cylinder of length $L_i$ and cross-section $\omega_\delta$  is denoted $\Omega_{i,\delta}=]0,L_i[\times \omega_\delta$.

\noindent \begin{definition}\label{DEF1} For all $i\in\{1,\ldots, N\}$, the  straight rod ${\cal P}_{i,\delta}$  is the cylinder of center line $\gamma_i$ and reference cross-section $\omega_\delta$.
We have ${\cal P}_{i,\delta}=\Phi_i\bigl(\Omega_{i,\delta}\bigr)$ where 
$$\Phi_i(s)=\varphi_i(s_i)+y_2\Gn_i+y_3\Gb_i=P_i+s_i\Gt_i+y_2\Gn_i+y_3\Gb_i,\qquad s=(s_i,y_2,y_3)\in\R^3.$$

\noindent The whole structure ${\cal S}_\delta$ is
\begin{equation*}
{\cal S}_\delta=\Big(\bigcup_{i=1}^N{\cal P}_{i,\delta}\Big)\cup\Big(\bigcup_{A\in \Gamma_{\cal K}} B\big(A; \delta\big)\Big).
\end{equation*}
\end{definition} Loosely speaking, in the definition of ${\cal S}_\delta$ a ball of radius $\delta$ is added at each knot  which belongs to $\Gamma_{\cal K}$. Remark that for a  knot $A\not\in\Gamma_{\cal K}$, the ball $B\big(A; \delta\big)$is  at least included in one ${\cal P}_{i,\delta}$.

There exists $\delta_0>0$ such that for $0<\delta\le \delta_0$ and for all  $(i,j)\in\{1,\ldots,N\}^2$, we have
$${\cal P}_{i,\delta}\cap{\cal P}_{j,\delta}=\emptyset\quad \hbox{if and only if}\quad\gamma_i\cap\gamma_j=\emptyset.$$

The reference domain associated to the straight  rod ${\cal P}_{i,\delta}$ is the open set $\Omega_i=]0,L_i[\times \omega$ (recall that $\omega$  is the disc of center the origin and radius 1). The running point of $\Omega_i$ (resp. $\Omega_{i,\delta}$, ${\cal S}_\delta$) is denoted $(s_i,Y_3,Y_3)$ (resp.  $(s_i,y_2,y_3)$,  $x$).
\subsection{The junctions.}
In what follows we deal with portions of the rod  ${\cal P}_{i,\delta}$. For any $h>0$ we set
\begin{equation*}
{\cal P}_{i,\delta}^{a,h}=\Phi_i\big(\Omega_{i,\delta}^{a,h}\big)\cap {\cal S}_\delta \quad \hbox{where}\quad \Omega_{i,\delta}^{a,h}=]a-h\delta,a+h\delta[\times \omega_\delta ,\qquad 0\le a \le L_i.
\end{equation*}
\vskip 1mm
 Let $A$ be a knot, for $h>0$ we consider the open set 
\begin{equation}\label{junction}
{\cal J}_{A, h\delta}=\bigcup_{i\in\{1,\ldots,N\},\; A\in\gamma_i} {\cal P}_{i,\delta}^{a^k_i,h}\qquad A=\varphi_i(a^k_i)\enskip\hbox{if }\; A\in \gamma_i.
\end{equation}

Up to choosing $\delta_0$ smaller, it is clear that there  exists  a real number 
$$\ds 1\le \rho_0\le {1\over 4\delta_0}\min_{(A,B)\in {\cal K}^2, A\not=B}||\overrightarrow{AB}||_2$$ depending on  ${\cal S}$ (via the angles between the segments of the skeleton ${\cal S}$) such that  for all $0<\delta\le \delta_0$

$\bullet$ $ {\cal J}_{A, (\rho_0+1)\delta}\cap  {\cal J}_{B, (\rho_0+1)\delta}=\emptyset$ for all distinct knots $A$ and $B$,

$\bullet $ the set $\ds {\cal S}_\delta \setminus \bigcup_{A\in {\cal K}}{\cal J}_{A, \rho_0\delta}$ is made by disjoined cylinders. 
\vskip 1mm
\noindent The junction in the neighborhood of $A$ is the domain ${\cal J}_{A, \rho_0\delta}$.

\noindent Notice that for any knot $A$ and for any $0<\delta\le \delta_0$
\begin{equation}\label {junctionTR}
\begin{aligned}
&\hbox{\it the domain ${\cal J}_{A, (\rho_0+1)\delta}$ is star-shaped with respect to the ball $B\big( A ; \delta\big)$}\\
&\hbox{\it and has a diameter less than $(2\rho_0+5)\delta$.}
 \end{aligned}\end{equation} 
\subsection{The functional spaces.}
\noindent For $q\in[1,+\infty]$, the $L^q$-class fields of ${\cal S}$ is the product space
\begin{equation*}
L^q({\cal S};\R^p)= \prod_{i=1}^N L^q(0,L_i;\R^p) \end{equation*} equiped   with the norm
\begin{equation*}
\begin{aligned}
&||V||_{L^q({\cal S};\R^p)}=\Big(\sum_{i=1}^N||V_i||^q_{L^q(0,L_i ;\R^p)}\Big)^{1/q},\quad V =(V_1,\ldots, V_N),\quad q\in[1,+\infty[,\\
&||V||_{L^{\infty}({\cal S};\R^p)}=\max_{i=1,\ldots,N}||V_i||_{L^{\infty}(0,L_i ;\R^p)}.
\end{aligned}
\end{equation*} 

\noindent The $W^{1,q}$-class fields of ${\cal S}$ make up a space denoted
\begin{equation*}
\begin{aligned}
W^{1,q}({\cal S};\R^p)=\Bigl\{V\in & \prod_{i=1}^N W^{1,q}(0,L_i;\R^p)\;| \;V =(V_1,\ldots, V_N),\;\;\hbox{such that }\\
& \hbox{for any } A\in {\cal K},\;\hbox{  if  }A\in\gamma_i\cap\gamma_j, \hbox{ with } A=\varphi_i(a_i^k)=\varphi_j(a_j^l),\\
&\hbox{ then one has }  V_i(a^k_i)=V_j(a^l_j)\Bigr\}.
\end{aligned}
 \end{equation*} The common value $V_i(a^k_i)$ is denoted $V(A)$. We equip $W^{1,q}({\cal S};\R^p)$ with the norm
\begin{equation*}
\begin{aligned}
&||V||_{W^{1,q}({\cal S};\R^p)}=\Big(\sum_{i=1}^N||V_i||^q_{W^{1,q}(0,L_i ;\R^p)}\Big)^{1/q},\qquad q\in[1,+\infty[,\\
&||V||_{W^{1,\infty}({\cal S};\R^p)}=\max_{i=1,\ldots,N}||V_i||_{W^{1,\infty}(0,L_i ;\R^p)}.
\end{aligned}
\end{equation*} 
  Indeed, the parametrization $\varphi=( \varphi_1,\ldots,\varphi_N)$ of ${\cal S}$ belongs to $W^{1,\infty}({\cal S} ;\R^3)$.
  
We also denote by $L^q({\cal S};SO(3))$ (respectively $W^{1,q}({\cal S};SO(3))$) the set of matrix fields $\GR$ in $L^q({\cal S};\R^{3\times 3})$ (resp. $W^{1,q}({\cal S};\R^{3\times 3})$) satisfying $\GR(s_i)\in SO(3)$ for almost any $s_i\in ]0,L_i[$, $i\in\{1,\ldots,N\}$.
\section{An approximation theorem.} 
\subsection{Definition of elementary rod-structure deformations.}
 We recall that for $x\in {\cal P}_{i,\delta}$ we have $x=\Phi_i(s)$ where $s\in \Omega_{i,\delta}$.
 
\begin{definition}\label{DEF2} An {\bf elementary rod-structure deformation} is a deformation $v_e$  verifying in each rod ${\cal P}_{i,\delta}$ ($i\in \{1,\ldots,N\}$) and each junction  ${\cal J}_{A,\rho_0\delta}$
\begin{equation}\label{defelem}
\begin{aligned}
v_e(s)&={\cal V}_i(s_i)+\GR_i(s_i)\big(y_2\Gn_i+y_3\Gb_i\big)\qquad s=(s_i,y_2,y_3)\in \Omega_{i,\delta}\\
v_e(x)&={\cal V}(A)+\GR(A)\big(x-A\big)\qquad x\in {\cal J}_{A,\rho_0\delta}\\
\end{aligned}
\end{equation}
where ${\cal V}\in H^{1}({\cal S};\R^3)$ and $\GR\in H^{1}({\cal S};SO(3))$ are such that  $v_e\in H^1({\cal
S}_\delta;\R^3)$. 
\end{definition}
\vskip 1mm
Recall that in the above definition ${\cal V}(A)$ (respectively $\GR(A)$) denote the common value of ${\cal V}_i$ (resp. $\GR_i$) at the knot $A$. Let us notice that in view of  Definition \ref{DEF2} and of  ${\cal J}_{A,\rho_0\delta}$ one has 
\begin{equation}\label{defelemJ}
\begin{aligned}
v_e(s)&={\cal V}(A)+(s_i-a^k_i)\GR(A)\Gt_i+\GR(A)\big(y_2\Gn_i+y_3\Gb_i\big),\\
 s&=(s_i,y_2,y_3)\in {\cal J}_{A,\rho_0\delta}\cap \Omega_{i,\delta},\quad A=\varphi_i(a^k_i),\\
{\cal V}_i(s_i)&={\cal V}(A)+(s_i-a^k_i)\GR(A)\Gt_i,\quad \GR_i(s_i)=\GR(A),\\
 s_i&\in ]a^k_i-\rho_0\delta, a^k_i+\rho_0\delta[\cap[0,L_i].
 \end{aligned}
\end{equation}  The field ${\cal V}_i$ stands for the deformation of the line $\gamma_i$ while $\GR_i(s_i)$ represents the rotation of the cross section with arc lenght $s_i$ on $\gamma_i$. Then Definition \ref{DEF2} impose to an elementary deformation to be a translation-rotation in  each junction.

\subsection {Decomposition of the deformation in each rod ${\cal P}_{i,\delta}$.}
According to \cite{BG2}, we first recall that   any deformation $v\in H^1({\cal P}_{i,\delta} ; \R^3 )$ can be decomposed as 
\begin{equation}\label{Decomp}
v(s)={\cal V}^{'}_i (s_i)+\GR^{'}_i(s_i)\big(y_2\Gn_i+y_3\Gb_i\big)+\overline{v}^{'}_i(s),\qquad s=(s_i,y_2,y_3)\in \Omega_{i,\delta},
\end{equation}
where ${\cal V}^{'}_i $ belongs to $H^1(0,L_i ; \R^3 )$, $\GR^{'}_i$  belongs to $H^1(0,L_i ; \R^{3\times 3})$ and satisfies for any $s_i\in [0,L_i]$: $\GR^{'}_i(s_i)\in SO(3)$ and $\overline{v}^{'}_i$ belongs to $H^1({\cal P}_{i,\delta} ; \R^3 ) $ (or $H^1(\Omega_{i,\delta} ; \R^3 ) $  using  again the same convention as for $v$). The term ${\cal V}^{'}_i$ gives the deformation of the center line of the rod. The second term $\GR^{'}_i(s_i)\big(y_2\Gn_i+y_3\Gb_i\big)$ describes the rotation of the cross section (of the   rod) which contains the point $\varphi_i(s_i)$. The part ${\cal V}^{'}_i (s_i)+\GR^{'}_i(s_i)\big(y_2\Gn_i+y_3\Gb_i\big)$ of the decomposition of $v$ is an elementary deformation of the rod ${\cal P}_{i,\delta}$. Let us notice that there is  no reason for ${\cal V}^{'}=\big({\cal V}^{'}_1,\ldots,{\cal V}^{'}_N\big)$ to belong to $H^1({\cal S};\R^3)$ or for  ${\GR}^{'}=\big(\GR^{'}_1,\ldots,\GR^{'}_N\big)$ to belong to $H^1({\cal S};\R^{3\times 3})$. This is why, in order to define an elementary deformation of the whole structure ${\cal S}_\delta$, we will modify the fields ${\cal V}^{'}_i$ and $\GR^{'}_i$ near to the junction ${\cal J}_{A, (\rho_0+1)\delta}$ at each knot $A$. In view of \eqref{DecN}, we choose the field $\GR_A(x-A)+{\bf a}_A$ as the elementary deformation in ${\cal J}_{A, \rho_0\delta}$.

The following theorem is proved in \cite{BG2}:
\smallskip
\begin{theorem}  Let $v\in H^1({\cal P}_{i,\delta} ; \R^3 ) $, there exists a decomposition \eqref{Decomp} such that 
\begin{equation}\label{DCP}
\begin{aligned}
&||\overline{v}^{'}_i||_{L^2(\Omega_{i,\delta} ; \R^3)}\le C\delta ||\hbox{dist}(\nabla_x v,SO(3))||_{L^2({\cal P}_{i,\delta})},\\
&||\nabla_s\overline{v}^{'}_i||_{L^2(\Omega_{i,\delta} ; \R^{3\times 3})}\le C ||\hbox{dist}(\nabla_x v,SO(3))||_{L^2({\cal P}_{i,\delta})},\\
&\Bigl\|{d\GR^{'}_i\over ds_i}\Big\|_{L^2(0,L_i ; \R^{3\times 3})}\le {C\over \delta^2} ||\hbox{dist}(\nabla_x v,SO(3))||_{L^2({\cal P}_{i,\delta})},\\
& \Bigl\|{d{\cal V}^{'}_i\over ds_i}-\GR^{'}_i \Gt_i\Big\|_{L^2(0,L_i ; \R^3)}\le {C\over \delta}||\hbox{dist}(\nabla_x v,SO(3))||_{L^2({\cal P}_{i,\delta})},\\
& \bigl\|\nabla_x v-\GR^{'}_i \big\|_{L^2(\Omega_{i,\delta};\R^{3\times 3})}\le C||\hbox{dist}(\nabla_x v,SO(3))||_{L^2({\cal P}_{i,\delta})},
\end{aligned}
\end{equation}  where the constant $C$ does not depend on $\delta$ and $L_i$.
\end{theorem}
\subsection {Decomposition of the deformation in each junction ${\cal J}_{A,\rho_0\delta}$.}
Let $v$ be a deformation belonging to $H^1({\cal S}_\delta ; \R^3)$ and let $A$ be a knot. We apply the rigidity theorem of \cite{FJM}, formulated in the version of Theorem  II.1.1 in \cite{BG2} which is licit because of \eqref{junctionTR}, to the domain  ${\cal J}_{A, (\rho_0+1)\delta}$. Hence there exist ${\bf R}_A\in SO(3)$ and ${\bf a}_A\in \R^3$ such that
\begin{equation}\label{DecN}
\begin{aligned}
&||\nabla_x v-{\bf R}_A||_{L^2({\cal J}_{A, (\rho_0+1)\delta} ; \R^{3\times 3})}\le C||\hbox{dist}(\nabla v, SO(3))||_{L^2({\cal J}_{A, (\rho_0+1)\delta})},\\
& ||v-{\bf a}_A-{\bf R}_A(x-A)||_{L^2({\cal J}_{A, (\rho_0+1)\delta} ; \R^{3})}\le C\delta||\hbox{dist}(\nabla v, SO(3))||_{L^2({\cal J}_{A, (\rho_0+1)\delta})},
 \end{aligned}
\end{equation}
with a constant $C$ which does not depend on $\delta$.
\vskip 1mm
\subsection{ The approximation theorem.}
  In Theorem \ref{The3.3} we show that any deformation  in  $H^1({\cal S}_\delta;\R^3)$  can be approximated by an elementary
rod-structure deformation $v_e\in H^1({\cal S}_\delta ; \R^3)$ of the type given in Definition  \ref {DEF2} 
\vskip 1mm
\begin{theorem}\label{The3.3}   Let $v$ be a deformation  in  $H^1({\cal S}_{\delta};\R^3)$. There exists an elementary  rod-structure deformation $v_e\in  H^1({\cal S}_\delta ; \R^3)$  in the sense of Definition \ref{DEF2} such that if we set   $\overline{v}=v-v_e$ 
\begin{equation}\label{Lem2}
\begin{aligned}
&||\overline{v}||_{L^2({\cal S}_{\delta} ; \R^3)}\le C\delta ||\hbox{dist}(\nabla_x v,SO(3))||_{L^2({\cal S}_{\delta})},\\
&||\nabla_x\overline{v}||_{L^2({\cal S}_{\delta} ; \R^{3\times 3})}\le C ||\hbox{dist}(\nabla_x v,SO(3))||_{L^2({\cal S}_{\delta})}.
\end{aligned}
\end{equation}
Moreover the fields ${\cal V}$ and $\GR$ associated to $v_e$ sastify
\begin{equation}\label{Lem1}
\begin{aligned}
&\sum_{i=1}^N\Bigl\|{d\GR_i\over ds_i}\Big\|_{L^2(0,L_i ; \R^{3\times 3})}\le  {C\over \delta^2}||\hbox{dist}(\nabla_x v,SO(3))||_{L^2({\cal S}_{\delta})},\\
& \sum_{i=1}^N\Bigl\|{d{\cal V}_i\over ds_i}-\GR_i\Gt_i\Big\|_{L^2(0,L_i ; \R^{3})}\le  {C\over \delta}||\hbox{dist}(\nabla_x v,SO(3))||_{L^2({\cal S}_{\delta})},\\
& \sum_{i=1}^N\bigl\|\nabla_x v-\GR_i \big\|_{L^2(\Omega_{i,\delta};\R^{3\times 3})}\le C||\hbox{dist}(\nabla_x v,SO(3))||_{L^2({\cal S}_{\delta})}.\end{aligned}
\end{equation}  In all the above estimates, the constant $C$ is independent on $\delta$ and on the lengths $L_i$ ($i\in\{1,\ldots,N\}$).
\end{theorem}

\begin{proof}

\noindent{\it Step 1. } In this step we compare the two decompositions of $v$ given in Subsections 3.1 and 3.2 in ${\cal J}_{A, (\rho_0+1)\delta}\cap {\cal P}_{i,\delta}$.

Let $A=\varphi_i(a_i^k)$ be a knot in ${\cal P}_{i,\delta}$. Using the last estimate of \eqref{DCP} and the first one in \eqref{DecN} we obtain
\begin{equation}\label{RR_A}
\begin{aligned}
\bigl\|\GR_A-\GR^{'}_i \big\|_{L^2(]a^k_i-(\rho_0+1)\delta,a^k_i+(\rho_0+1)\delta[\cap ]0,L_i[ ; \R^{3\times 3})}\le {C\over \delta}\big(&||\hbox{dist}(\nabla_x v,SO(3))||_{L^2({\cal P}_{i,\delta})}\\
+||\hbox{dist}(&\nabla v, SO(3))||_{L^2({\cal J}_{A, (\rho_0+1)\delta})}\big).
\end{aligned}
\end{equation}
Now  the decomposition \eqref{Decomp}, the  first estimate in \eqref{DCP} and  the last estimate \eqref{DecN} lead to 
\begin{equation*}
\begin{aligned}
||{\cal V}^{'}_i +\GR^{'}_i\big(y_2\Gn_i+y_3\Gb_i\big)-{\bf R}_{A}\big((s_i-a^k_i)\Gt_i+y_2\Gn_i+y_3\Gb_i\big)-{\bf a}_{A}||_{L^2(\Omega_{i,a^k_i,(\rho_0+1)\delta}\cap\Omega_{i,\delta} ; \R^3)}\\ 
\le C\delta\big(||\hbox{dist}(\nabla_x v,SO(3))||_{L^2({\cal P}_{i,\delta})}+||\hbox{dist}(\nabla v, SO(3))||_{L^2({\cal J}_{A, (\rho_0+1)\delta})}\big)
\end{aligned}
\end{equation*}
which in turn with \eqref{RR_A} gives
\begin{equation}\label{Va}
\begin{aligned}
&||{\cal V}^{'}_i-{\bf a}_A -(s_i-a^k_i){\bf R}_{A}\Gt_i||_{L^2(]a^k_i-(\rho_0+1)\delta,a^k_i+(\rho_0+1)\delta[\cap]0,L_i[ ; \R^3)}\\ 
\le &C\big(||\hbox{dist}(\nabla_x v,SO(3))||_{L^2({\cal P}_{i,\delta})}+||\hbox{dist}(\nabla v, SO(3))||_{L^2({\cal J}_{A, (\rho_0+1)\delta})}\big).
\end{aligned}
\end{equation}

\noindent{\it Step 2. } Here we construct  $v_e$. 

We first define a  rotation field $\GR\in H^1({\cal S} ; SO(3))$ which is constant and equal to $\GR_A$ given by \eqref{DecN} in the adequate neighborhood of the  knot $A$ and which is close to $\GR^{'}_i$ on each segment $\gamma_i$. Let  $i\in \{1,\ldots, N\}$ be fixed. For all $k\in\{1,\ldots,K_i\}$ (i.e. all the knots of the line $\gamma_i$) we first set 
\begin{equation}\label{GR1}
\GR_i(s_i)=\GR_{A}\qquad A=\varphi_i(a^k_i)\in {\cal K}\cap \gamma_i,\quad s_i\in ]a^k_i-\rho_0\delta,a^k_i+\rho_0\delta[,
\end{equation}
 and then ''far from the knots'' \begin{equation}\label{GR2}
\GR_i(s_i)=\GR^{'}_i (s_i)\qquad s_i\in ]0,L_i[\setminus \bigcup_{k=1}^{K_i}]a^k_i-(\rho_0+1)\delta,a^k_i+(\rho_0+1)\delta[.
\end{equation} It remains to define $\GR_i$ in the  intervals 
$$I^k_{+}=]0,L_i[\cap ]a^k_i+\rho_0\delta,a^k_i+(\rho_0+1)\delta[\enskip\hbox{and}\enskip I^k_{-}=]0,L_i[\cap ]a^k_i-(\rho_0+1)\delta,a^k_i-\rho_0\delta[,\enskip k\in\{1,\ldots,K_i\}.$$ In 
these intervals, we proceed as in Appendix of \cite{BG2} in order to construct the field $\GR_i$ by a sort of interpolation in $SO(3)$, from $\GR_A$ to $\GR^{'}_i(a^k_i+(\rho_0+1)\delta)$ and from $\GR_A$ to $\GR^{'}_i(a^k_i-(\rho_0+1)\delta)$. This construction satisfies 
\begin{equation}\label{GRI}
\begin{aligned}
\Big\|{d\GR_i\over ds_i}\Big\|^2_{L^2(I^k_+)^{3\times 3}}&\le {C\over \delta}|||\GR^{'}_i(a^k_i+(\rho_0+1)\delta)-\GR_A|||^2,\\
\Big\|{d\GR_i\over ds_i}\Big\|^2_{L^2(I^k_-)^{3\times 3}}&\le {C\over \delta}|||\GR_A-\GR^{'}_i(a^k_i-(\rho_0+1)\delta)|||^2.\\
\end{aligned}
\end{equation} The constant does not depend on $\delta$. From the above construction we obtain a field $\GR=\big(\GR_1,\ldots, \GR_N\big)$ defined on ${\cal S}$  and which  belongs to $H^1({\cal S} ; SO(3))$. 
\vskip 2mm
Let us now define the field ${\cal V}\in H^1({\cal S} ; \R^3)$ of $v_e$. We proceed as for $\GR$. Let  $i\in \{1,\ldots, N\}$ be fixed. For all $k\in\{1,\ldots,K_i\}$  we first set 
\begin{equation}\label{DefV1}
\begin{aligned}
{\cal V}_i(s_i)=&{\bf a}_{A}+(s_i-a^k_i){\bf R}_{A}\Gt_i\qquad A=\varphi_i(a^k_i)\in {\cal K}\cap \gamma_i,\\&\quad\hbox{for}\enskip s_i\in ]a^k_i-\rho_0\delta,a^k_i+\rho_0\delta[,
\end{aligned}
\end{equation} while ''far from the knots'' we set 
\begin{equation}\label{DefV2}
{\cal V}_i(s_i)={\cal V}^{'}_i (s_i)\qquad s_i\in ]0,L_i[\setminus \bigcup_{k=1}^{K_i}]a^k_i-(\rho_0+1)\delta,a^k_i+(\rho_0+1)\delta[.
\end{equation}

At least in the remaining intervals $I^k_+$ and $I^k_-$, we just perform a linear interpolation   
\begin{equation}\label{cal V}
\begin{aligned}
{\cal V}_i(s_i)&={a_i^k+\rho_0\delta-s_i\over \delta}{\cal V}^{'}_i (s_i)+\Big(1-{a_i^k+\rho_0\delta-s_i\over \delta}\Big)\Big({\bf a}_{A}+(s_i-a^k_i){\bf R}_{A}\Gt_i\Big)\\
&\quad \hbox{for}\enskip s_i\in I^k_+,\\
{\cal V}_i(s_i)&={a_i^k-\rho_0\delta-s_i\over \delta}{\cal V}^{'}_i (s_i)+\Big(1-{a_i^k-\rho_0\delta-s_i\over \delta}\Big)\Big({\bf a}_{A}+(s_i-a^k_i){\bf R}_{A}\Gt_i\Big)\\
&\quad \hbox{for}\enskip   s_i\in I^k_-,\qquad  k\in\{1,\ldots,K_i\}.
\end{aligned}
\end{equation} Gathering \eqref{DefV1},   \eqref{DefV2} and  \eqref{cal V}, we obtain a field ${\cal V}=\big({\cal V}_1,\ldots, {\cal V}_N\big)$ defined on  ${\cal S}$ which  belongs to $H^1({\cal S} ; \R^3)$. It worth noting that $({\cal V},\GR)$ verify the condition \eqref {defelemJ}. As a consequence we  can define the elementary deformation $v_e$ associated to ${\cal V}$ and $\GR$ through Definition \ref{DEF2} and it belongs to $H^1({\cal S}_\delta;\R^3)$.

\noindent{\it Step 3. } Comparison between $({\cal V}_i,\GR_i)$ and $({\cal V}_i', \GR_i')$.

Using the third estimate  in \eqref{DCP}, \eqref{RR_A} and \eqref{GRI} we obtain
\begin{equation*}
\begin{aligned}
&\bigl\|\GR_i-\GR^{'}_i \big\|_{L^2(I^k_{+}\cup I^k_{-} ; \R^{3\times 3})}+\delta\Bigl\|{d\GR_i\over ds_i}\Big\|_{L^2(I^k_{+}\cup I^k_{-}; \R^{3\times 3})}\\
\le & {C\over \delta}\big(||\hbox{dist}(\nabla_x v,SO(3))||_{L^2({\cal P}_{i,\delta})}
+||\hbox{dist}(\nabla v, SO(3))||_{L^2({\cal J}_{A, (\rho_0+1)\delta})}\big).
\end{aligned}
\end{equation*} Taking into account the definition of $\GR_i$, for all $i\in\{1,\ldots,N\}$ we finally get 
\begin{equation}\label{EstGR}
\begin{aligned}
&\bigl\|\GR_i-\GR^{'}_i \big\|_{L^2(0,L_i ; \R^{3\times 3})}+\delta\Bigl\|{d\GR_i\over ds_i}\Big\|_{L^2(0, L_i ; \R^{3\times 3})}\\
\le & {C\over \delta}\big(||\hbox{dist}(\nabla_x v,SO(3))||_{L^2({\cal P}_{i,\delta})}
+\sum_{A\in \gamma_i}||\hbox{dist}(\nabla v, SO(3))||_{L^2({\cal J}_{A, (\rho_0+1)\delta})}\big).
\end{aligned}\end{equation} The constant does not depend on $\delta$ and $L_i$.

From \eqref{Va} and the definition of ${\cal V}_i$ (see \eqref{DefV1}, \eqref{DefV2} and \eqref{cal V}) we deduce that
\begin{equation}\label{Va2}
\begin{aligned}
||{\cal V}^{'}_i -{\cal V}_i||_{L^2(0,L_i ; \R^3)}\le &C\big(||\hbox{dist}(\nabla_x v,SO(3))||_{L^2({\cal P}_{i,\delta})}\\
+& \sum_{A\in \gamma_i}||\hbox{dist}(\nabla v, SO(3))||_{L^2({\cal J}_{A, (\rho_0+1)\delta})}\big).
\end{aligned}
\end{equation}  Now, taking into account the fourth estimate in \eqref{DCP}, estimates  \eqref{RR_A}-\eqref{Va} and again the definition of ${\cal V}_i$ we get
\begin{equation}\label{EDVmoinsVp}
\begin{aligned}
\Bigl\|{d{\cal V}_i\over ds_i}-{d{\cal V}^{'}_i\over ds_i}\Big\|_{L^2(0,L_i ; \R^3)}\le &{C\over \delta}\big(||\hbox{dist}(\nabla_x v,SO(3))||_{L^2({\cal P}_{i,\delta})}\\
&+\sum_{A\in \gamma_i}||\hbox{dist}(\nabla v, SO(3))||_{L^2({\cal J}_{A, (\rho_0+1)\delta})}\big).
\end{aligned}\end{equation}  where the constant $C$ does not depend on $\delta$ and $L_i$.

\noindent{\it Step 4. }  
First of all, we prove the estimates \eqref{Lem1}. The first one in \eqref{Lem1} is a direct consequence of \eqref{EstGR} and the fact that the matrices $\GR_i$ are constants in the neighborhood of the knots (intervals $I^k_\pm$). The second one comes from the fourth estimate in  \eqref{DCP},  again \eqref{EstGR} and \eqref{EDVmoinsVp}. Then, from \eqref{DCP} and \eqref{EstGR} we deduce the last estimate in   \eqref{Lem1}.

Now, let us set $\overline{v}=v-v_e$ where $v_e$ is defined in Step 2.  From the decomposition \eqref{Decomp} and the expression of  $v_e$ in Definition \ref{DEF2}, we have for all $i\in\{1,\ldots,N\}$ 
\begin{equation*}
\begin{aligned}
& v(s)={\cal V}^{'}_i(s_i)+\GR^{'}_i(s_i)\bigl(y_2\Gn_i +y_3 \Gb_i\bigr)+\overline{v}^{'}_i(s),\\
& v(s)={\cal V}_i(s_i)+\GR_i(s_i)\bigl(y_2\Gn_i +y_3 \Gb_i\bigr)+\overline{v}_i(s)=v_{e,i}(s)+\overline{v}_i(s),\end{aligned}
\qquad s=(s_i,y_2,y_3)\in\Omega_{i,\delta}.
\end{equation*} Then, using \eqref{EstGR} and \eqref{Va2} it leads to
\begin{equation*}
\begin{aligned}
||\overline{v}_i-\overline{v}^{'}_i||_{L^2(\Omega_{i,\delta} ; \R^3)}\le &C\delta \big(||\hbox{dist}(\nabla_x v,SO(3))||_{L^2({\cal P}_{i,\delta})}\\
+& \sum_{A\in \gamma_i}||\hbox{dist}(\nabla v, SO(3))||_{L^2({\cal J}_{A, (\rho_0+1)\delta})}\big).
\end{aligned}
\end{equation*} Hence, due to the first estimate in \eqref{DCP} and the above inequalities we get
\begin{equation*}
\sum_{i=1}^N||\overline{v}_i||_{L^2(\Omega_{i,\delta} ; \R^3)}\le C\delta \big(||\hbox{dist}(\nabla_x v,SO(3))||_{L^2({\cal S}_{i\delta})}.
\end{equation*} 
Now, in order to take into account the knots in $\Gamma_{\cal K}$, we use \eqref{DecN} and the definition of $v_e$ in the junction ${\cal J}_{A, \rho_0\delta}$  (\eqref{GR1} and \eqref{DefV1}) to obtain
\begin{equation*}
\begin{aligned}
||v-v_e||_{L^2(B(A ; \delta) ; \R^{3})}\le ||v-v_e||_{L^2({\cal J}_{A, \rho_0\delta} ; \R^{3})}\le C\delta||\hbox{dist}(\nabla v, SO(3))||_{L^2({\cal J}_{A, (\rho_0+1)\delta})}.\end{aligned}
\end{equation*} Finally due to the definition of ${\cal S}_\delta$  we deduce the first estimate in \eqref{Lem2}.

\noindent An easy calculation gives in $\Omega_{i,\delta}$, $i=1,\ldots,N$
\begin{equation}\label{GradVE}
\nabla_x v_e\Gt_i={d{\cal V}_i\over ds_i}+{d\GR_i\over ds_i}(y_2\Gn_i+y_3\Gb_i),\qquad \nabla_x v_e\Gn_i=\GR_i\Gn_i,\qquad \nabla_x v_e\Gb_i=\GR_i\Gb_i.
\end{equation} Then,  from the first estimate in \eqref{DecN} and all those in  \eqref{Lem1}  we obtain the   estimate of $\nabla\overline{v}$.
\end{proof}

\section {A Korn's type inequality for the rod-structure.}

This section is devoted to derive a nonlinear Korn's type inequality for ${\cal S}_\delta$. We assume that the structure ${\cal S}_\delta$ is clamped on a few extremities whose set is denoted by  $\Gamma_0^\delta$, corresponding to a set $\Gamma_0$ of extremities of ${\cal S}$.
 Then  we set 
$$\D_\delta = \{v\in H^1({\cal S}_\delta; \R^3)\; |\; v=I_d \enskip\hbox{on} \enskip \Gamma_0^\delta\}.$$
Let $v$ be in $\D_\delta$. Proceeding as in \cite{BG2} for each rod, we can choose $v_e$ in Theorem \ref{The3.3} such that 
\begin {equation}\label{clv_e}
v_e=I_d\quad \hbox{on}\quad \Gamma_0^\delta.
\end{equation}
Notice that \eqref{clv_e} implies that 
\begin {equation}\label{clvbar}
\overline{v}=0\enskip \hbox{on}\enskip \Gamma_0^\delta,\qquad\quad{\cal V}=\phi,\enskip \GR_i=\GI_3 \enskip \hbox{on}\enskip \Gamma_0\hbox{ for } i=1,\ldots,N.
\end{equation}
\begin{theorem}\label{Korn} There exists a constant $C$ which depends only on ${\cal S}$ such that, for all deformation $v\in \D_\delta$ 
\begin {equation}\label{EstmKorn}
 ||v-I_d||_{H^1({\cal S}_\delta ; \R^{3})}\le   {C\over \delta}||\hbox{dist}(\nabla_x v,SO(3))||_{L^2({\cal S}_{\delta})}.
\end{equation}
Moreover
\begin {equation}\label{EstmKornN}
\begin{aligned}
\sum_{A\in {\cal K}}\; \; ||{\cal V}(A)-A||_2 \le {C\over \delta^2}||\hbox{dist}(\nabla_x v,SO(3))||_{L^2({\cal S}_{\delta})},\\
\sum_{A\in {\cal K}}\; \; |||\GR(A)-\GI_3||| \le {C\over \delta^2}||\hbox{dist}(\nabla_x v,SO(3))||_{L^2({\cal S}_{\delta})}.
\end{aligned}
\end{equation}
\end{theorem}
\begin{proof}
Taking into account the connexity of ${\cal S}$ and the continuous character of $\GR$, the boundary condition \ref{clvbar} implies that 
\begin {equation*}
\sum_{i=1}^N||\GR_i-\GI_3||_{L^2(0,L_i ; \R^{3\times 3})}\le C\sum_{i=1}^N\Bigl\|{d\GR_i\over ds_i}\Big\|_{L^2(0,L_i ; \R^{3\times 3})}.
\end{equation*} Then, the first estimate in \eqref{Lem1} gives
\begin {equation}\label{EstRL2}
\sum_{i=1}^N||\GR_i-\GI_3||_{L^2(0,L_i ; \R^{3\times 3})}\le  {C\over \delta^2}||\hbox{dist}(\nabla_x v,SO(3))||_{L^2({\cal S}_{\delta})}.
\end{equation}
From \eqref{EstRL2}, the last estimate in \eqref{Lem1}, the boundary condition on $v$ together with Poincar\' e inequality we deduce that \ref{EstmKorn} holds true.
Now estimate \eqref{EstRL2} and the second one in \eqref{Lem1} imply  that
\begin {equation}\label{EstdVL2}
 \sum_{i=1}^N\Bigl\|{d{\cal V}_i\over ds_i}-\Gt_i\Big\|_{L^2(0,L_i ; \R^{3})}\le   {C\over \delta^2}||\hbox{dist}(\nabla_x v,SO(3))||_{L^2({\cal S}_{\delta})}.
\end{equation} Hence the boundary condition \eqref{clvbar} leads to
\begin {equation}\label{EstVL2}
||{\cal V}-\varphi||_{H^1({\cal S} ; \R^{3})}\le  {C\over \delta^2}||\hbox{dist}(\nabla_x v,SO(3))||_{L^2({\cal S}_{\delta})}
\end{equation} from which the first estimate in \eqref{EstmKornN} follows. At last the last estimate in \eqref{EstmKornN} comes from \eqref{Lem1} and the boundary condition on $\GR$.
The constants depend only on the skeleton ${\cal S}$.  
\end{proof}
\begin{remark} Since $\GR$ is bounded, we also have (for $i=1,\ldots,N$)
\begin {equation}\label{EstR3}
||\GR_i-\GI_3||_{L^2(0,L_i ; \R^{3\times 3})}\le C. \end{equation}  From the above estimate and proceeding as in the proof of Theorem \ref{Korn} we obtain
\begin {equation}\label{EstmKorn2}
 ||v-I_d||_{H^1({\cal S}_\delta ; \R^{3})}\le  C\Big\{\delta+ ||\hbox{dist}(\nabla_x v,SO(3))||_{L^2({\cal S}_{\delta})}\Big\},
\end{equation} 
and 
\begin {equation}\label{EstmKornN2}
\begin{aligned}
&\sum_{A\in {\cal K}}\; \; ||{\cal V}(A)-A||_2 \le C\Big\{1+{1\over \delta} ||\hbox{dist}(\nabla_x v,SO(3))||_{L^2({\cal S}_{\delta})}\Big\},\\
&\sum_{A\in {\cal K}}\; \; |||\GR(A)-\GI_3||| \le C.
\end{aligned}
\end{equation}

\end{remark}
\begin{remark}\label{pourforce}
From \eqref{Va2} and   estimate (II.3.9) in \cite{BG2} and the fact that $(v-I_d)-({\cal V}_{i}-\varphi_i)=(\GR_i-\GI_3)(y_2\Gn_i+y_3\Gb_i)+\overline{v}_i$ in each rod ${\cal P}_{i,\delta}$, we obtain
$$||(v-I_d)-({\cal V}_i-\varphi_i)||_{L^2({\cal P}_{i,\delta} ; \R^3)}\le  C ||\hbox{dist}(\nabla_x v,SO(3))||_{L^2({\cal P}_{i,\delta})}.$$
Due to \eqref{EstRL2} and the estimate (II.3.10) in \cite{BG2},  we also have
\begin {equation}\label{tenslin}
\begin{aligned}
\bigl\|\nabla_x v+(\nabla_x v)^T-2\GI_3\big\|_{L^2({\cal S}_\delta ; \R^{3\times 3})}\le& C||\hbox{dist}(\nabla_x v,SO(3))||_{L^2({\cal S}_\delta)}\\
+&{C\over \delta^3}||\hbox{dist}(\nabla_x v,SO(3))||^2_{L^2({\cal S}_\delta)},
\end{aligned}
\end{equation}
which in turn with \eqref{Lem1} give
\begin{equation}\label{SymPartR}
\begin{aligned}
\sum_{i=1}^N||\GR_i^T+\GR_i-2\GI_3||_{L^2(0,L_i ; \R^{3\times 3})}\le & {C\over \delta}||\hbox{dist}(\nabla_x v,SO(3))||_{L^2({\cal S}_\delta)}\\
 + & {C\over \delta^4}||\hbox{dist}(\nabla_x v,SO(3))||^2_{L^2({\cal S}_\delta)}.
\end{aligned}
\end{equation}
\end{remark}
\section    {Rescaling of $\Omega_{i,\delta}$ for $i=1,\ldots,N$.}

\noindent We recall that $\Omega_i=]0,L_i[\times \omega$. We rescale $\Omega_{i,\delta}$ using the operator 
$$\Pi_{i,\delta} (\psi)(s_i,Y_2,Y_3)=\psi(s_i,\delta Y_2,\delta Y_3)\quad\hbox{ for a.e.}\quad (s_i,Y_2,Y_3) \in \Omega_i$$
defined for any measurable function $\psi$ over $\Omega_{i,\delta}$. Indeed, if $\psi\in L^2(\Omega_{i,\delta})$ then $\Pi_{i,\delta} (\psi)\in  L^2(\Omega_i)$.
The estimates \eqref{Lem2} of $\overline{v}$   transposed over $\Omega_i$ are 
\begin {equation}\label{estpivbar}
\begin{aligned}
||\Pi_{i,\delta}( \overline{v})||_{L^2(\Omega_i ; \R^3)}\le C ||\hbox{dist}(\nabla_x v,SO(3))||_{L^2({\cal S}_\delta)},\\
\Big\|{\partial \Pi_{i,\delta}(\overline {v})\over \partial Y_2}\Big\|_{L^2(\Omega_i ; \R^3)}\le C ||\hbox{dist}(\nabla_x v,SO(3))||_{L^2({\cal S}_\delta)},\\ 
\Big\|{\partial \Pi_{i,\delta}(\overline{v})\over \partial Y_3}\Big\|_{L^2(\Omega_i ; \R^3)}\le C ||\hbox{dist}(\nabla_x v,SO(3))||_{L^2({\cal S}_\delta)}, \\ 
\Big\|{\partial \Pi_{i,\delta}(\overline{v})\over \partial s_i}\Big\|_{L^2(\Omega_i ; \R^3)}\le {C \over \delta} ||\hbox{dist}(\nabla_x v,SO(3))||_{L^2({\cal S}_\delta)}.\\
\end{aligned}
\end{equation}
\section{Asymptotic behavior of the Green-St Venant's strain tensor.}

We distinguish three main cases for the behavior of the quantity  $||\hbox{dist}(\nabla_x v,SO(3))||_{L^2({\cal S}_\delta)}$ which plays an important role in the derivation of estimates in nonlinear elasticity
\begin{equation*}
||\hbox{dist}(\nabla_x v,SO(3))||_{L^2({\cal S}_\delta)}=\left\{\begin{aligned}
&O(\delta^{\kappa}),\qquad 1\le \kappa<2,\\
&O(\delta^2),\\
&O(\delta^{\kappa}),\qquad \kappa>2.
\end{aligned}
\right.
\end{equation*}
This hierarchy of behavior for $||\hbox{dist}(\nabla_x v,SO(3))||_{L^2({\cal P}_\delta)}$ has already been observed in \cite{BG2}.  

In this section we investigate the behavior of a sequence $(v_\delta)_\delta$ of $\D_\delta $ for which  $||\hbox{dist}(\nabla_x v_\delta,SO(3))||_{L^2({\cal S}_\delta)}=O(\delta^{\kappa})$ for $1<\kappa\le 2$.   In Subsection \ref{Subcase1} we analyse the case $1<\kappa<2$ and   Subsection \ref{Subcase2} deals with the case $\kappa=2$. Let us emphasize that we explicit the limit of the Green-St Venant's tensor in each case. 
\subsection{Case  $1<\kappa<2$. Limit  behavior for  a sequence such that  $||\hbox{dist}(\nabla_x v_\delta,SO(3))||_{L^2({\cal S}_\delta)} \sim \delta^\kappa$.}\label{Subcase1}
 
Let us consider a sequence of deformations $(v_\delta)_\delta$ of $\D_\delta $  such that for $1<\kappa<2$
\begin{equation*}
||\hbox{dist}(\nabla_x v_\delta,SO(3))||_{L^2({\cal S}_\delta)} \le C \delta^\kappa.
\end{equation*}
We denote by ${\cal V}_\delta$, $\GR_\delta$ and $\overline{v}_\delta$ the three terms of the decomposition of $v_\delta$ given by Theorem \ref{The3.3}.  The two first estimates of \eqref{Lem1} and those in \eqref{estpivbar}  lead to the following lemma:
\begin{lemma}\label{case1}
There exists a subsequence still indexed by $\delta$ such that 
\begin{equation}\begin{aligned}\label{weakconv1}
&\GR_\delta \rightharpoonup \GR \quad \hbox{weakly-* in}\quad  L^\infty({\cal S} ; \R^{3\times 3}),\\
&{1\over \delta ^{\kappa-2}}\GR_\delta \rightharpoonup  0 \quad \hbox{weakly in}\quad H^1({\cal S} ; \R^{3\times 3}),\\
&{\cal V}_\delta\rightharpoonup  {\cal V} \quad \hbox{weakly in}\quad  H^1({\cal S} ; \R^3),\\
&{1\over \delta^{\kappa-1}}\GR_{i,\delta}^T\big({d{\cal V}_{i,\delta}\over ds_i}-\GR_{i,\delta}\Gt_i\big)\rightharpoonup   {\cal Z}_i \quad \hbox{weakly  in}\quad  L^2(0,L_i ; \R^3),\\
&{1\over \delta^\kappa}\Pi_{i,\delta}\big(\GR_{i,\delta}^T\overline {v}_{\delta}\big) \rightharpoonup  \overline{w}_i \quad \hbox{weakly in}\quad  L^2(0,L_i;H^1(\omega ; \R^3)).
\end{aligned}
\end{equation}
Moreover the matrix $\GR_i(s_i)$ belongs to the convex hull of  $SO(3)$ for almost any $s_i \in ]0,L_i[$, $i\in\{1,\ldots, N\}$, ${\cal V}\in W^{1,\infty}({\cal S} ; \R^3) $ and they satisfy 
\begin{equation}\label{condlimi1}
{\cal V}=\phi \quad  \hbox{ on }\quad \Gamma_0 \qquad\hbox{and} \qquad{d{\cal V}_i \over ds_i}=\GR_i \Gt_i\qquad i\in\{1,\ldots, N\}.
\end{equation}
\noindent Furthermore, we also have
\begin{equation*}
\begin{aligned}
\Pi_{i,\delta}( v_\delta)&\rightharpoonup     {\cal V}_i \quad \hbox{weakly in}\quad H^1(\Omega_i ; \R^3),\\
\Pi_{i,\delta} (\nabla_x v_\delta)& \rightharpoonup  \GR_i\quad \hbox{weakly in}\quad L^2(\Omega_i ; \R^{3\times 3}).\\
\end{aligned}
\end{equation*}
\end{lemma}
Now we proceed as in \cite{BG2} to derive the limit of the Green-St Venant tensor  as $\delta$ goes to $0$. We first recall that for any function $\psi\in H^1(\Omega_{i,\delta})$
$${\partial \psi\over \partial s_i}=\nabla_x \psi\;\Gt_i\qquad {\partial \psi\over \partial y_2}=\nabla_x \psi\;\Gn_i\qquad {\partial \psi\over \partial y_3}=\nabla_x \psi \Gb_i.$$
Then in view of Lemma \ref{case1} and of the above relation, we obtain 
\begin{equation}\label{PrTGSV}
\begin{aligned}
{1\over \delta^{\kappa-1}}\GR_{i,\delta}^T\big(\Pi_{i,\delta} (\nabla_x v_\delta)-\GR_{i,\delta}\big)\Gt_i& \rightharpoonup  {\cal Z}_i \quad \hbox{weakly in}\quad L^2(\Omega_i ; \R^3),\\
{1\over \delta^{\kappa-1}}\GR_{i,\delta}^T\big(\Pi_{i,\delta} (\nabla_x v_\delta)-\GR_{i,\delta}\big)\Gn_i& \rightharpoonup  {\partial\overline{w}_i\over \partial Y_2}\quad \hbox{weakly in}\quad L^2(\Omega_i ; \R^3),\\
{1\over \delta^{\kappa-1}}\GR_{i,\delta}^T\big(\Pi_{i,\delta} (\nabla_x v_\delta)-\GR_{i,\delta}\big)\Gb_i& \rightharpoonup  {\partial\overline{w}_i\over \partial Y_3}\quad \hbox{weakly in}\quad L^2(\Omega_i ; \R^3),\\
\end{aligned}
\end{equation} 
 The weak convergences in \eqref{PrTGSV} together with the relation 
$(\nabla_xv_\delta)^T\nabla_x v_\delta-\GI_3=(\nabla_xv_\delta-\GR_{i,\delta})^T\GR_{i,\delta}+(\GR_{i,\delta})^T(\nabla_x v_\delta-\GR_{i,\delta})+(\nabla_xv_\delta-\GR_{i,\delta})^T(\nabla_x v_\delta-\GR_{i,\delta})$ permit to obtain the limit of the Green-St Venant's tensor in the rescaled domain $\Omega_i$. We obtain
\begin{equation}\label{GSVcase1}
{1\over 2\delta^{\kappa-1}}\Pi_{i,\delta} \big((\nabla_xv_\delta)^T\nabla_x v_\delta-\GI_3\big)\rightharpoonup   \GE_i \qquad\hbox{weakly in}\quad L^1(\Omega_i ; \R^{3\times 3}),
\end{equation} with
\begin{equation}\label{GSVEcase1}
\begin{aligned}
\GE_i=&{1\over 2}\Big\{(\Gt_i\,|\, \Gn_i\,|\, \Gb_i) \Bigl({\cal Z}_i\,|\,{\partial\overline{w}_i\over \partial Y_2}\,|\, {\partial\overline{w}_i\over \partial Y_3}\Big)^T
+\Bigl({\cal Z}_i\,|\,{\partial\overline{w}_i\over \partial Y_2}\,|\, {\partial\overline{w}_i\over \partial Y_3}\Big)(\Gt_i\,|\, \Gn_i\,|\, \Gb_i)^T\Big\}.
\end{aligned}
\end{equation} The term $\big(\Ga\, |\, \Gb\, |\, \Gc\big)$ denotes the $3\times 3$ matrix with  columns   $\Ga$,  $\Gb$ and  $\Gc$.
\medskip
For any field $\overline{\psi}\in L^2(0, L_i ; H^1(\omega ; \R^3))$ we set
\begin{equation}\label{evbar}
\begin{aligned}
 e_{22}(\overline{\psi})&={\partial\overline{\psi}\over \partial Y_2}\cdot\Gn_i ,\quad e_{23}(\overline{\psi})={1\over 2}\Big\{{\partial\overline{\psi}\over \partial Y_2}\cdot\Gb_i+{\partial\overline{\psi}\over \partial Y_3}\cdot\Gn_i\Big\},\quad e_{33}(\overline{\psi})={\partial\overline{\psi}\over \partial Y_3}\cdot\Gb_i.
\end{aligned}
\end{equation}
Hence we can write $\GE_i$ as
$$\GE_i=(\Gt_i\,|\, \Gn_i\,|\, \Gb_i) \, \widehat{\GE}_i\,(\Gt_i\,|\, \Gn_i\,|\, \Gb_i) ^T$$ where the symmetric matrix $\widehat{\GE}_i$ is defined by  
\begin{equation}\label{HatEcase1}
\widehat{\GE}_i= \begin{pmatrix}
\ds{\cal Z}_i\cdot\Gt_i & * &*  \\   \\
\ds {1\over 2}{\partial\overline{u}_i\over \partial Y_2}\cdot\Gt_i & \ds e_{22}(\overline{u}_i) & * \\   \\
 \ds{ 1\over 2}{\partial\overline{u}_i\over \partial Y_3}\cdot\Gt_i & \ds e_{23}(\overline{u}_i) & \ds e_{33}(\overline{u}_i)
 \end{pmatrix}
 \end{equation} and where the field $\overline{u}_i\in L^2(0,L_i;H^1(\omega ; \R^3))$   is defined by
\begin{equation}\label{ubarcase1}
\overline{u}_i=\big[Y_2\big({\cal Z}_i\cdot\Gn_i\big)+Y_3\big({\cal Z}_i\cdot\Gb_i\big)\big]\Gt_i+\overline{w}_i,\qquad i\in\{1,\ldots,N\}.
\end{equation}

\subsection{Case  $\kappa=2$. Limit  behavior for  a sequence such that  $||\hbox{dist}(\nabla_x v_\delta,SO(3))||_{L^2({\cal S}_\delta)} \sim \delta^2$. }\label{Subcase2}

Let us consider a sequence of deformations $(v_\delta)_\delta$ of $\D_\delta$  such that
$$||\hbox{dist}(\nabla_x v_\delta,SO(3))||_{L^2({\cal S}_\delta)} \le C \delta^2.$$ Indeed, in each rod ${\cal P}_{i,\delta}$ we get $||\hbox{dist}(\nabla_x v_\delta,SO(3))||_{L^2({\cal P}_{i,\delta})} \le C \delta^2$. 
We denote by ${\cal V}_\delta$, $\GR_\delta$ and $\overline{v}_\delta$ the three terms of the decomposition of $v_\delta$ given by Theorem \ref{The3.3}.  The estimates \eqref{Lem1}, \eqref{Lem2} and \eqref{EstVL2} lead to the following lemma:

\begin{lemma}\label{Lem6.2}
There exists a subsequence still indexed by $\delta$ such that 
\begin{equation}\label{case2}
\begin{aligned}
&\GR_\delta \rightharpoonup \GR \quad \hbox{weakly in}\quad  H^1({\cal S} ; SO(3))\enskip\hbox{and strongly in }\enskip L^\infty({\cal S} ; SO(3)),\\
&{\cal V}_\delta\longrightarrow   {\cal V} \quad \hbox{strongly in}\quad  H^1({\cal S} ; \R^3),\\
&{1\over \delta}\big({d{\cal V}_{i,\delta}\over ds_i}-\GR_{i,\delta}\Gt_i\big)\rightharpoonup   {\cal Z}_i \quad \hbox{weakly  in}\quad  L^2(0,L_i ; \R^3),\\
&{1\over \delta^2}\Pi_{i,\delta}\big(\overline {v}_{\delta}\big) \rightharpoonup  \overline{v}_i \quad \hbox{weakly in}\quad  L^2(0,L_i ; H^1(\omega ; \R^3)).
\end{aligned}
\end{equation} Moreover ${\cal V}_i\in H^2(0,L_i; \R^3)$ for all $i\in\{1,\ldots, N\}$ and we have
\begin{equation}\label{condlimi2}
{\cal V}=\phi, \quad  \GR_i=\GI_3,\quad\hbox{on}\quad \Gamma_0\quad\hbox{and} \quad{d{\cal V}_i\over ds_i}=\GR_i\Gt_i\qquad i\in\{1,\ldots, N\}.
\end{equation}
\noindent Furthermore, we also have
\begin{equation}\label{Conv-v}
\begin{aligned}
\Pi_{i,\delta} (v_\delta)&\longrightarrow    {\cal V}_i \quad \hbox{strongly in}\quad  H^1(\Omega_i ; \R^3),\\
\Pi_{i,\delta} (\nabla_x v_\delta)& \longrightarrow  \GR_i\quad \hbox{strongly in}\quad L^2(\Omega_i ; \R^{3\times 3}).
\end{aligned}
\end{equation}
\end{lemma}
As a consequence of the above lemma we obtain
\begin{equation}\label{ConvGradcase2}
\begin{aligned}
{1\over \delta}\big(\Pi_{i,\delta} (\nabla_x v_\delta)-\GR_{i,\delta}\big)\Gt_i &\rightharpoonup  {d\GR_i\over ds_i}\big(Y_2\Gn_i+Y_3\Gb_i)+{\cal Z}_i \quad \hbox{weakly in}\quad 
L^2(\Omega_i ; \R^3),\\
{1\over \delta}\big(\Pi_{i,\delta} (\nabla_x v_\delta)-\GR_{i,\delta}\big)\Gn_i& \rightharpoonup  {\partial\overline{v}_i\over \partial Y_2}\quad \hbox{weakly in}\quad L^2(\Omega_i ; \R^3),\\
{1\over \delta}\big(\Pi_{i,\delta} (\nabla_x v_\delta)-\GR_{i,\delta}\big)\Gb_i& \rightharpoonup  {\partial\overline{v}_i\over \partial Y_3}\quad \hbox{weakly in}\quad L^2(\Omega_i ; \R^3).
\end{aligned}
\end{equation}

Proceeding as in Subsection  \ref{Subcase1} and using convergences \eqref{case2} and \eqref{ConvGradcase2}
 permit to obtain the limit of the Green-St Venant's tensor in the rescaled domain $\Omega_i$
\begin{equation}\label{GSVcase2}
{1\over 2\delta}\Pi_{i,\delta} \big((\nabla_xv_\delta)^T\nabla_x v_\delta-\GI_3\big)\rightharpoonup   \GE_i \qquad\hbox{weakly in}\quad  L^1(\Omega_i; \R^{3\times 3}),
\end{equation} with
\begin{equation}\label{ConvSTV}
\begin{aligned}
\GE_i=&{1\over 2}\Big\{(\Gt_i\,|\, \Gn_i\,|\, \Gb_i) \Bigl({d\GR_i\over ds_i}\big(Y_2\Gn_i+Y_3\Gb_i)+{\cal Z}_i\,|\,{\partial\overline{v}_i\over \partial Y_2}\,|\, {\partial\overline{v}^{i)}\over \partial Y_3}\Big)^T\GR_i\\
+& \GR^T_i \Bigl({d\GR_i\over ds_i}\big(Y_2\Gn_i+Y_3\Gb_i)+{\cal Z}_i\,|\,{\partial\overline{v}_i\over \partial Y_2}\,|\, {\partial\overline{v}^{i)}\over \partial Y_3}\Big)(\Gt_i\,|\, \Gn_i\,|\, \Gb_i)^T\Big\}.
\end{aligned}
\end{equation}
Setting 
\begin{equation}\label{ubarcase2}
\overline{u}_i=\big[Y_2\big({\cal Z}_i\cdot\GR_i\Gn_i\big)+Y_3\big({\cal Z}_i\cdot\GR_i\Gb_i\big)\big]\Gt_i+\GR^T_i\overline{v}_i,\qquad i\in\{1,\ldots,N\}.
\end{equation} and using the fact that the matrix $\ds \GR^T_i{d\GR_i\over ds_3}$ is antisymmetric, we can write $\GE_i$ as
\begin{equation*}
\GE_i=(\Gt_i\,|\, \Gn_i\,|\, \Gb_i)\, \widehat{\GE}_i\,(\Gt_i\,|\, \Gn_i\,|\, \Gb_i)^T,
\end{equation*} where the symmetric matrix $\widehat{\GE}_i$ is defined by  
\begin{equation}\label{HatEcase2}
\widehat{\GE}_i= \begin{pmatrix}
\ds-Y_2\Gamma_{i,3}(\GR)+Y_3\Gamma_{i,2}(\GR)+{\cal Z}_i\cdot\GR_i\Gt_i & * & * \\ \\
\ds -{1\over 2}Y_3\Gamma_{i,1}(\GR)+{1\over 2}{\partial\overline{u}_i\over \partial Y_2}\cdot\Gt_i & \ds e_{22}(\overline{u}_i) & * \\ \\
\ds {1\over 2}Y_2\Gamma_{i,1}(\GR)+{1\over 2}{\partial\overline{u}_i\over \partial Y_3}\cdot\Gt_i &  \ds e_{23}(\overline{u}_i) &\ds e_{33}(\overline{u}_i) \end{pmatrix}
 \end{equation} where
\begin{equation}\label{GammaZcase2}
\Gamma_{i,3}(\GR)={d\GR_i\over ds_i}\Gt_i\cdot \GR_i\Gn_i,\qquad \Gamma_{i,2}(\GR)={d\GR_i\over ds_i}\Gb_i\cdot \GR_i\Gt_i,\qquad \Gamma_{i,1}(\GR)={d\GR_i\over ds_i}\Gn_i\cdot \GR_i\Gb_i,
\end{equation} and where the $e_{kl}(\overline{u}_i)$'s are given  by \eqref{evbar}.

\section{Elastic structure}
  \medskip
In this section  we assume that  the  structure ${\cal S}_\delta$ is made of an elastic material. The associated local energy $W\ \ :\ \  \GS_3\longrightarrow \R^+$ is a  continuous function of symmetric matrices which satisfies the following assumptions 
\begin{equation}\label{HypW}
\begin{aligned}
&\exists c>0 \quad \hbox{such  that }\quad \forall E\in\GS_3\quad  W(E)\ge c|||E|||^2,\\
&\forall\varepsilon >0,\quad \exists \theta >0,\quad \hbox{such  that }\\
&\forall E\in\GS_3\quad |||E|||\le \theta\; \Longrightarrow \; |W(E)-Q(E)|\le \varepsilon |||E|||^2,
\end{aligned}
\end{equation} where $Q$ is a positive quadratic form defined on the set of $3\times 3$ symmetric matrices (see e.g. \cite{BG4}). Remark that $Q$ satisfies the first inequality in \eqref{HypW} with the same constant $c$. 

Following \cite{C1}, for any $3\times3$ matrix $F$, we set 
\begin{equation}\label{W}
\widehat{W}(F)=\left\{\begin{aligned}
&W\Big({1\over 2}(F^TF-\GI_3)\Big) \quad\hbox{if} \quad \det(F)>0,\\
&+\infty \hskip 2.8cm \hbox{if} \quad \det(F)\le 0.
\end{aligned}\right.
\end{equation}

Remark that due to \eqref{HypW} and to the inequality $|||F^TF-\GI_3|||\ge dist(F,SO(3))$ if $\det (F)>0$, we have for any $3\times 3$ matrix $F$
\begin{equation}\label{HatW}
\widehat{W}(F)\ge {c\over 4} dist (F,SO(3))^2. 
\end{equation}
 
\begin{remark}\label{7.1}  A classical example of  local elastic energy satisfying the above assumptions is given by the St Venant-Kirchhoff's density (see \cite{C1}) 
\begin{equation}\label{HatW2}
\widehat{W}(F)=\left\{\begin{aligned}
&{\lambda\over 8}\big(tr(F^TF-\GI_3) \big)^2+{\mu\over 4}tr\big((F^TF-\GI_3)^2\big)\quad\hbox{if}\quad \det(F)>0,\\
&+\infty\hskip 6.6cm \hbox{if}\quad \det(F)\le 0.
\end{aligned}\right.
\end{equation}
 The coefficients $\lambda$ and $\mu$ are the Lam's constants. In this case we have for all matrix $E\in \GS_3$
 $$Q(E)={\lambda\over 2}\big(tr(E) \big)^2+{\mu}\,tr\big(E^2\big).$$
  \end{remark}

Now we assume that the structure ${\cal S}_\delta$ is submitted to applied body forces $f_{\kappa,\delta}\in L^2({\cal S}_\delta ; \R^3)$ and we define   the total energy  
$J_{\kappa,\delta}(v)$\footnote{For later convenience, we have added the term $\ds \int_{{\cal S}_\delta}f_{\kappa,\delta}(x)\cdot  I_d(x)dx$
to the usual standard energy, indeed this does not affect the  minimizing problem for  $J_{\kappa,\delta}$. }  over $\D_\delta$ by
\begin{equation}\label{J}
J_{\kappa,\delta}(v)=\int_{{\cal S}_\delta}\widehat{W}(\nabla_x v)(x)dx-\int_{{\cal S}_\delta}f_{\kappa,\delta}(x)\cdot (v(x)-I_d(x))dx.
\end{equation}

\noindent {\bf  Assumptions on the forces.}{\it  We set 
\begin{equation}\label{kappaprime}
\kappa^{'}=\left\{\begin{aligned}
&2\kappa-2\quad \hbox{if }\enskip 1\le\kappa\le 2,\\
&\kappa\hskip 14mm \hbox{if }\enskip \kappa\ge 2.
\end{aligned}\right.
\end{equation} 
In what follows we define the forces applied to the structure by distinguish the forces applied to the junctions and to their complementary in ${\cal S}_\delta$. Moreover in order to take into account the resultant and the moment of the forces near each knot  and in each cross section of the rods we decompose the forces  density into two types. The first one mainly works with  the mean deformation ${\cal V}$ while the second one is related to the rotation $\GR$.  We begin by the definition of the forces in the junctions.

\noindent For any knot $A$, let $F_A$ and  $G_A$ be  two fields belonging to $L^2({\cal J}_{A,\rho_0}; \R^3)$\footnote{The domain ${\cal J}_{A,\rho_0}$ is obtained by transforming ${\cal J}_{A,\rho_0\delta}$ by a dilation of center $A$ and ratio $1/\delta$. }, the second field $G_A$ satisfying $\ds \int_{{\cal J}_{A,\rho_0}} G_A(z)dz=0$. 

We define the applied forces in  the junction ${\cal J}_{A,\rho_0\delta}$   ($A\in {\cal K}$) by
\begin{equation}\label{ForceN}
f_{\kappa,\delta}(x)=\delta^{\kappa^{'}-1} F_A\Big(A+{x-A\over \delta}\Big)+\delta^{\kappa^{'}-2} G_A\Big(A+{x-A\over \delta}\Big),\qquad \hbox{ for a.e. }  x\in {\cal J}_{A,\rho_0\delta}.
\end{equation} 

In order   to precise the forces $f_{\delta,\kappa}$ in the complementary of the junctions (i.e. in ${\cal S}_\delta\setminus  \bigcup_{A\in {\cal K}}{\cal J}_{A,\rho_0\delta}$), we follow \cite{BG2} and we assume that there exist $f$, $g^{(\Gn)}$ and $g^{(\Gb)}$ in $ L^2({\cal S} ; \R^3)$ such that
\begin{equation}\label{ForceP}
\begin{aligned}
f_{\kappa,\delta}(x)&=
\delta^{\kappa^{'}} f_i(s_i)+\delta^{\kappa^{'}-2}\big(y_2 g^{(\Gn)}_i(s_i)+y_3g^{(\Gb)}_i
(s_i)\big),\qquad x=\Phi_i(s)\\
& \hbox{ for a.e.}\enskip s\in\big(]0,L_i[\setminus \bigcup_{k=1}^{K_i}]a_i^k-\rho_0\delta,a_i^k+\rho_0\delta[\big)\times \omega_\delta 
\end{aligned}
\end{equation} }
Notice that $J_{\kappa,\delta}(I_d)=0$. So,  in order to minimize $J_{\kappa,\delta}$ we only need to  consider  deformations $v $ of $\GD_{\delta}$ such that   $J_{\kappa,\delta}(v )\le 0$. 
Now from the decomposition given in Theorem \ref{The3.3}, and the definition \eqref{ForceN}-\eqref{ForceP} of the forces we first get
\begin {equation} \label{estforces}
\begin{aligned}
|\int_{{\cal S}_\delta}f_{\kappa,\delta}(x)\cdot (v(x)-I_d(x))dx| \le C\delta^{\kappa^{'}+2}\Big[||f||_{L^2({\cal S};\R^3)}||{\cal V}-\phi||_{L^2({\cal S};\R^3)}&\\
+ (||g^{(\Gn)}||_{L^2({\cal S};\R^3)}+||g^{(\Gb)}||_{L^2({\cal S};\R^3)})\Big(\sum_{i=1}^N||\GR_i-\GI_3||_{L^2(0,L_i;\R^{3\times3})}\Big)&\\
+\sum_{A\in {\cal K}}||F_A||_{L^2({\cal J}_{A,\rho_0}; \R^3)}||{\cal V}(A)-A||_2+\sum_{A\in {\cal K}}||G_A||_{L^2({\cal J}_{A,\rho_0}; \R^3)}|||\GR(A)-\GI_3|||&\Big].
\end{aligned}\end{equation}
In the case $\kappa\ge 2$, using \eqref{EstmKornN}, \eqref{EstRL2} and \eqref{EstVL2} gives
\begin {equation} \label{estforces1}
\begin{aligned}
|\int_{{\cal S}_\delta}f_{\kappa,\delta}(x)\cdot (v(x)-I_d(x))dx| \le C\delta^{\kappa}||\hbox{dist}(\nabla_x v , SO(3))||_{L^2({\cal S}_\delta)}
\end{aligned}\end{equation}
while in the case $1\le \kappa\le 2$, \eqref{EstR3} and \eqref{EstmKornN2} lead to 
\begin {equation} \label{estforces2}
\begin{aligned}
|\int_{{\cal S}_\delta}f_{\kappa,\delta}(x)\cdot (v(x)-I_d(x))dx| \le C\delta^{2\kappa}\Big\{1+{1\over \delta}||\hbox{dist}(\nabla_x v , SO(3))||_{L^2({\cal S}_\delta)}\Big\}
\end{aligned}\end{equation}
where in both cases the constant $C$ depends on the $L^2$-norms of $f$, $g^{(\Gn)}$, $g^{(\Gb)}$, $F_A$ and $G_A$ (and are then independent upon $\delta$).
For a deformation $v$ such that $J_{\kappa,\delta}(v )\le 0$, the coerciveness assumption \eqref{HatW} and the estimates \eqref{estforces1} and \eqref{estforces2} allows us to obtain for $1\le \kappa$
\begin{equation}\label{EstDist}
||\hbox{dist}(\nabla_x v , SO(3))||_{L^2({\cal S}_\delta)}\le C\delta^{\kappa}.\end{equation} Again  estimates \eqref{estforces1} and \eqref{estforces2} and \eqref{EstDist} lead to 
\begin{equation}\label{EstJ}
c\delta^{2\kappa}\le J_{\kappa,\delta}(v)\le 0.\end{equation}
with a constant independent on $\delta$.

\noindent Then, from  \eqref{HypW}-\eqref{W}-\eqref{HatW} and  the estimates  \eqref{estforces1}, \eqref{estforces2} and \eqref{EstDist}  we deduce
\begin{equation*}
{c\over 4}||(\nabla_x v )^T\nabla_x v -\GI_3||^2_{L^2({\cal S}_\delta ; \R^{3\times 3})} \le J_{\kappa,\delta}(v)+\int_{{\cal S}_\delta}f_{\kappa,\delta}\cdot(v -I_d)\le C\delta^{2\kappa}.
\end{equation*} Hence, the following estimate of the Green-St Venant's tensor hold true:
$$\big\|{1\over 2}\big\{(\nabla_x v )^T\nabla_x v -\GI_3\big\}\big\|_{L^2({\cal S}_\delta; \R^{3\times 3})}\le  C\delta^{\kappa}.$$ We deduce from the above inequality that $v\in (W^{1,4}({\cal S}_\delta))^3$  with
$$||\nabla_x v ||_{L^4({\cal S}_\delta; \R^{3\times 3})}\le  C\delta^{1\over 2}.$$

We set
$$m_{\kappa,\delta}=\inf_{v\in \D_\delta}J_{\kappa,\delta}(v)$$
 and we recall that, in general, a minimizer of $J_{\kappa,\delta}$ does not exist  on $\D_\delta$.

\section{Asymptotic behavior  of $m_{\kappa,\delta}$ for $1<\kappa<2$.}

The goal of this section is to establish  Theorem \ref{theo8.1} below. Let us first introduce a few notations. We denote by ${\cal C}$ the convex hull of the set $SO(3)$
\begin{equation}
{\cal C}=\overline{conv} (SO(3)).
\end{equation}
We set 
\begin{equation}\label{deflim}
\begin{aligned}
\V\R_1=\Big\{({\cal V},\GR)\in H^1({\cal S} ; \R^3) \times L^2({\cal S} ; {\cal C})\;|\; 
{\cal V} & =\phi  \quad  \hbox{ on }\quad \Gamma_0,\\
{d{\cal V}_i \over ds_i} & =\GR_i \Gt_i \quad i\in\{1,\ldots, N\}\Big\}.
\end{aligned}\end{equation}
We define the linear functional  ${\cal L}$ over $H^1({\cal S} ; \R^3) \times L^2({\cal S} ; \R^{3\times 3})$ by 
\begin{equation}\label {L1}
\begin{aligned}
{\cal L}\big({\cal V},\GR\big)&=\sum_{i=1}^N\pi\int_{0}^{L_i}\Big( f_i\cdot \big({\cal V}_i-\phi_i\big) +{g^{(\Gn)}_i
\over 3}\cdot\big(\GR_i-\GI_3)\Gn_i+{g^{(\Gb)}_i\over 3}\cdot\big(\GR_i-\GI_3)\Gb_i\Big)ds_i\\
 +\sum_{A\in {\cal K}}\Big[\Big(&\int_{{\cal J}_{A,\rho_0}} F_A(y) dy\Big)\cdot\big({\cal V}(A)-\phi(A)\big)+\int_{{\cal J}_{A,\rho_0}} G_A(y)\cdot\big(\GR(A)-\GI_3\big)(y-A)\, dy\Big].
\end{aligned}
\end{equation}  It is easy to prove that the infimum of $-{\cal L}$ on $\V\R_1$ is a minimum.

\begin {theorem}\label{theo8.1}
We have 
\begin{equation}\label{res1}
\lim_{\delta\to 0}{m_{\kappa,\delta}\over \delta^{2\kappa}}=\min_{({\cal V},\GR)\in \V\R_1}\big(-{\cal L}\big({\cal V},\GR\big)\big).
\end{equation}
\end{theorem}
\begin{proof}
{\it Step 1.} In this step we show that 
$$\min_{({\cal V},\GR)\in \V\R_1}\big(-{\cal L}\big({\cal V},\GR\big)\big)\le \liminf_{\delta\to 0}{m_{\kappa,\delta}\over \delta^{2\kappa}}.$$

Let   $(v_\delta)_\delta$  be  a sequence of deformations belonging to $\D_\delta$ and such that
\begin{equation}\label{Hypvdelta}
\lim_{\delta\to 0}{J_{\kappa,\delta}(v_\delta)\over \delta^{2\kappa}}= \liminf_{\delta\to 0}{m_{\kappa,\delta}\over \delta^{2\kappa}}.
\end{equation} 
We can always assume that $J_{\kappa,\delta}(v_\delta)\le 0$.
Then, from the estimates of the previous section we obtain
\begin{equation*}
\begin{aligned}
&||\hbox{dist}(\nabla v_\delta, SO(3))||_{L^2({\cal S}_\delta)}\le C\delta^{\kappa},\\
& \big\|{1\over 2}\big\{\nabla v_\delta^T\nabla v_\delta-\GI_3\big\}\big\|_{L^2({\cal S}_\delta ; \R^{3\times 3})}\le  C\delta^\kappa,\\
 &||\nabla v_\delta||_{L^4({\cal S}_\delta ; \R^{3\times 3})}\le  C\delta^{1\over 2}.
 \end{aligned}
 \end{equation*} For any fixed $\delta$, the deformation $v_\delta$ is decomposed as in Theorem \ref{The3.3} and we are in a position to apply the results of  Subsection \ref{Subcase1}. There exists a subsequence still indexed by $\delta$ such that 
 \begin{equation}\begin{aligned}\label{8weakconv1}
&\GR_\delta \rightharpoonup \GR^{\{0\}} \quad \hbox{weakly-* in}\quad  L^\infty({\cal S} ; \R^{3\times 3}),\\
&{1\over \delta ^{\kappa-2}}\GR_\delta \rightharpoonup  0 \quad \hbox{weakly in}\quad H^1({\cal S} ; \R^{3\times 3}),\\
&{\cal V}_\delta\rightharpoonup  {\cal V}^{\{0\}} \quad \hbox{weakly in}\quad  H^1({\cal S} ; \R^3),\\
&{1\over \delta^{\kappa-1}}\GR_{i,\delta}^T\big({d{\cal V}_{i,\delta}\over ds_i}-\GR_{i,\delta}\Gt_i\big)\rightharpoonup   {\cal Z}^{\{0\}}_i \quad \hbox{weakly  in}\quad  L^2(0,L_i ; \R^3),\\
&{1\over \delta^\kappa}\Pi_{i,\delta}\big(\GR_{i,\delta}^T\overline {v}_{\delta}\big) \rightharpoonup  \overline{w}^{\{0\}}_i \quad \hbox{weakly in}\quad  L^2(0,L_i;H^1(\omega ; \R^3)).
\end{aligned}
\end{equation}
The couple $({\cal V}^{\{0\}},\GR^{\{0\}})$ belongs to $\V\R_1$.
\noindent Furthermore, we also have ($i\in\{1,\ldots,N\}$)
\begin{equation}\label{(8.7)}
\begin{aligned}
\Pi_{i,\delta} v_\delta&\rightharpoonup     {\cal V}^{\{0\}}_i \quad \hbox{weakly in}\quad H^1(\Omega_i ; \R^3),\\
\Pi_{i,\delta} (\nabla_x v_\delta)& \rightharpoonup  \GR^{\{0\}}_i\quad \hbox{weakly in}\quad L^4(\Omega_i ; \R^{3\times 3}),\\
{1\over 2\delta^{\kappa-1}}\Pi_{i,\delta} \big((\nabla_xv_\delta)^T\nabla_x v_\delta- & \GI_3\big)\rightharpoonup   \GE^{\{0\}}_i \quad\hbox{weakly in}\quad L^2(\Omega_i ; \R^{3\times 3}),
\end{aligned}
\end{equation}
 where the symmetric matrix ${\GE}^{\{0\}}_i$ is defined in \eqref{GSVEcase1} (see  Subsection \ref{Subcase1}). 
Due to the decomposition of $v_\delta$ and the above convergences \eqref{8weakconv1} 
we immediately get the limit of the term involving the body forces
$$ \lim_{\delta\to 0}{1\over \delta^{2\kappa}}\int_{{\cal S}_\delta}f_{\kappa,\delta}\cdot(v_\delta -I_d) = {\cal L}\big({\cal V}^{\{0\}},\GR^{\{0\}}\big)$$ where $ {\cal L}\big({\cal V},\GR\big)$ is defined by \eqref{L1}. 

\noindent Recall that we have $\ds-\int_{{\cal S}_\delta}f_{\kappa,\delta}\cdot(v_\delta -I_d)\le J_{\kappa,\delta}(v_\delta)$. Then, due to  \eqref{Hypvdelta} and the above limit, we finally get
\begin{equation}\label{Step11}
\min_{({\cal V},\GR)\in \V\R_1}\big(-{\cal L}\big({\cal V},\GR\big)\big)\le -{\cal L}\big({\cal V}^{\{0\}},\GR^{\{0\}}\big)\le\liminf_{\delta\to 0}{m_{\kappa,\delta}\over \delta^{2\kappa}}.
\end{equation}

\noindent {\it Step 2.} In this step we show that 
$$\limsup_{\delta\to 0}{m_{\kappa,\delta}\over \delta^{2\kappa}}\le \min_{({\cal V},\GR)\in \V\R_1}\big(-{\cal L}\big({\cal V},\GR\big)\big).$$ 
Let $({\cal V}^{\{1\}},\GR^{\{1\}})\in \V\R_1$ such that $\ds-{\cal L}\big({\cal V}^{\{1\}},\GR^{\{1\}}\big)=\min_{({\cal V},\GR)\in \V\R_1}\big(-{\cal L}\big({\cal V},\GR\big)\big)$.

\noindent  Using Proposition \ref{PropAprox} in the Appendix , there exists a sequence  $({\cal V}^{(n)},\GR^{(n)})_{n\ge 0}$ in $ \V\R_1$ which satisfies

$\bullet$  $\GR^{(n)}\in W^{1,\infty}({\cal S};SO(3))$,

$\bullet$  $\GR^{(n)}$ is equal to $\GI_3$ in a neighbourhood  of each knot $A\in {\cal K}$ and each  fixed extremity belonging to $\Gamma_0$,

$\bullet$ 
\begin{equation}\label{(8.8)}
\begin{aligned}
{\cal V}^{(n)} \rightharpoonup {\cal V}^{\{1\}} \quad \hbox{weakly  in} \;H^1({\cal S};\R^3),\\
\GR^{(n)}\rightharpoonup \GR^{\{1\}} \quad \hbox{weakly in} \; L^2({\cal S};{\cal C}).
\end{aligned}
\end{equation}

Now we fix $n$. Since $({\cal V}^{(n)},\GR^{(n)})$ in $ \V\R_1$ and due to the second condition  imposed on $\GR^{(n)}$ above, we can   consider the elementary deformation $v^{(n)}$ constructed by using  $({\cal V}^{(n)},\GR^{(n)})$ in  Definition \ref{DEF2}. Indeed the deformation $v^{(n)}$ belongs to $\D_\delta\cap W^{1,\infty}({\cal S}_\delta;\R^3)$. Thanks  to the expression of the gradient of $v^{(n)}$ (see \eqref{GradVE}) and to the definition of $\V\R_1$ we have
\begin{equation*}\label{GradVN10}
||\nabla_x v^{(n)}-\GR^{(n)}_{i}||_{L^\infty(\Omega_{i,\delta} ; \R^{3\times 3})}\le C_n\delta.
\end{equation*} Now using the identity $\big(\nabla_x v^{(n)}\big)^T\nabla_x v^{(n)}-\GI_3=\big(\nabla_x v^{(n)}-\GR^{(n)}_{i}\big)^T\GR^{(n)}_{i}+\big(\GR^{(n)}_{i}\big)^T\big(\nabla_x v^{(n)}-\GR^{(n)}_{i}\big)+\big(\nabla_x v^{(n)}-\GR^{(n)}_{i}\big)^T\big( \nabla_x v^{(n)}-\GR^{(n)}_{i}\big)$ and the above estimate, we obtain
\begin{equation}\label{GradVN11}
{1\over 2\delta^{\kappa-1}}\Pi_{i,\delta}\Big(\big(\nabla_x v^{(n)}\big)^T\nabla_x v^{(n)}-\GI_3\big)\Big)\longrightarrow 0\quad \hbox{strongly in } L^\infty(\Omega_i ; \R^{3\times 3})
\end{equation} as $\delta \to 0$. Hence, if $\delta$ is small enough, we get $\det\big(\nabla_x v^{(n)}\big)>0$ a.e. in  ${\cal S}_\delta$. 

In what follows, we show that $\ds\lim_{\delta \to 0}{1\over \delta^{2\kappa}} J_{\kappa,\delta}(v^{(n)})=-{\cal L}\big({\cal V}^{(n)},\GR^{(n)}\big)$.
Let ${\cal P}_{i,\delta}$ be a rod of the structure. Using the third assumption in \eqref{HypW} (with $\varepsilon=1$) and the estimate \eqref{GradVN11}, for $\delta$ small enough  we have
$${1\over \delta^{2\kappa}} \int_{{\cal P}_{i,\delta}}\widehat{W}(\nabla_x v^{(n)})(x)dx \le {1\over \delta^{2(\kappa-1)}} \int_{\Omega_{i}}\big(Q\big(E(v^{(n)})\big)+|||E(v^{(n)})|||^2\big)ds_idY_2dY_3$$ where 
$$E(v^{(n)})=\Pi_{i,\delta}\Big(\big(\nabla_x v^{(n)}\big)^T\nabla_x v^{(n)}-\GI_3\big)\Big).$$ Thanks to the convergence \eqref{GradVN11} we obtain
$\ds {1\over \delta^{2\kappa}} \int_{{\cal P}_{i,\delta}}\widehat{W}(\nabla_x v^{(n)})(x)dx\to 0$ as $\delta\to 0$. Notice that since $v^{(n)}$ is an elementary deformation, its Green-St Venant's tensor is null in the neightbourhood of the knots. Finally we get
$${1\over \delta^{2\kappa}} \int_{{\cal S}_{\delta}}\widehat{W}(\nabla_x v^{(n)})(x)dx\longrightarrow 0\quad\hbox{as $\delta$ tends to $0$.}$$ Using again the fact that $v^{(n)}$ is an elementary deformation, and assumptions \eqref{ForceN} and \eqref{ForceP} on the forces, we immediately get the limit of the term involving the body forces
$$ \lim_{\delta\to 0}{1\over \delta^{2\kappa}}\int_{{\cal S}_\delta}f_{\kappa,\delta}\cdot(v^{(n)} -I_d) ={\cal L}\big({\cal V}^{(n)},\GR^{(n)}\big).$$ Indeed, since $v^{(n)}\in \D_\delta$, we have
$${m_{\kappa,\delta}\over \delta^{2\kappa}}\le {J_{\kappa,\delta}(v^{(n)})\over \delta^{2\kappa}}.$$ Passing to the limit as $\delta$ tends to 0 we obtain
$$\limsup_{\delta\to 0}{m_{\kappa,\delta}\over \delta^{2\kappa}}\le-{\cal L}\big({\cal V}^{(n)},\GR^{(n)}\big).$$  In view of the convergences \eqref{(8.8)} we are able to pass to the limit as $n$ tends to infinity and we obtain
$$\limsup_{\delta\to 0}{m_{\kappa,\delta}\over \delta^{2\kappa}}\le-{\cal L}\big({\cal V}^{\{1\}},\GR^{\{1\}}\big)= \min_{({\cal V},\GR)\in \V\R_1}\big(-{\cal L}\big({\cal V},\GR\big)\big).$$ This concludes the proof of the theorem.
\end{proof}
\begin{remark} Let us point out that Theorem \ref{theo8.1} shows that for any minimizing sequence $(v_\delta)_\delta$ as in Step 1, the third convergence of the rescaled Green-St Venant's strain tensor in \eqref{(8.7)} is a strong convergence to $0$ in $L^2(\Omega_i;\R^{3\times 3})$.
\end{remark}

\section{Asymptotic behavior  of $m_{2,\delta}$.}

The goal of this section is to establish  Theorem \ref{theo9.1} below. Let us first introduce a few notations. We set 
\begin{equation}\label{deflim2}
\begin{aligned}
\V\R_2=\Big\{({\cal V},\GR)\in H^1({\cal S} ; \R^3) \times H^1({\cal S} ; SO(3))\;|\; 
{\cal V} & =\phi,\enskip \GR_i=\GI_3 \quad  \hbox{ on }\quad \Gamma_0,\\
{d{\cal V}_i \over ds_i} & =\GR_i \Gt_i \quad i\in\{1,\ldots, N\}\Big\}.
\end{aligned}\end{equation}

\begin {theorem}\label{theo9.1}
We have 
\begin{equation}\label{res2}
\lim_{\delta\to 0}{m_{2,\delta}\over \delta^{4}}=\min_{({\cal V},\GR)\in \V\R_2} {\cal J}_2\big({\cal V},\GR\big),
\end{equation}
where the functional ${\cal J}_2$ is defined by 
\begin{equation}\label {J2}
{\cal J}_2\big({\cal V},\GR\big)=\sum_{i=1}^N \int_0^{L_i} a_{lk}\Gamma_{i,k}(\GR)\Gamma_{i,l}(\GR)-{\cal L}({\cal V},\GR).
\end{equation} In ${\cal J}_2$,  the $3\times 3$ matrix $\GA=(a_{ij})$ is symmetric and definite positive. This matrix depend on $\omega$ and on the quadratic form $Q$.
\end{theorem}
 Note that the infimum of  $ {\cal J}_2$ on $\V\R_2$ is actually a minimum.
 
 For a St-Venant- Kirchhoff material, whose energy is recalled in Remark  \ref{7.1}, the expression of the matrix $\GA$ is explicitly derived  at the end of the  appendix (see Remark \ref{10.7})  and it leads to the following limit energy

\begin{equation}
{\cal J}_2\big({\cal V},\GR\big)={\pi\over 4}\sum_{i=1}^N \int_0^{L_i} \big(\mu|\Gamma_{i,1}(\GR)|^2+E|\Gamma_{i,2}(\GR)|^2+E|\Gamma_{i,3}(\GR)|^2\big) -{\cal L}({\cal V},\GR)
\end{equation}
where $E$ and $\mu$ are respectively the Young and Poisson's coefficients.
\begin{proof}[Proof of Theorem \ref {theo9.1}]
{\it Step 1.} In this step we show that 
$$ \min_{({\cal V},\GR)\in \V\R_2} {\cal J}_2\big({\cal V},\GR\big)\le \liminf_{\delta\to 0}{m_{2,\delta}\over \delta^{4}}.$$

Let   $(v_\delta)_\delta$  be  a sequence of deformations belonging to $\GD_\delta$ and such that
\begin{equation}\label{Hypvdelta2}
\lim_{\delta\to 0}{J_{2,\delta}(v_\delta)\over \delta^{4}}=\liminf_{\delta\to 0}{m_{2,\delta}\over \delta^{4}}.
\end{equation} We can always assume that $J_{\kappa,\delta}(v_\delta)\le 0$.
Then, from the estimates of the Section 7 we obtain
\begin{equation}\label{GSV2}
\begin{aligned}
&||\hbox{dist}(\nabla v_\delta, SO(3))||_{L^2({\cal S}_\delta)}\le C\delta^{2},\\
& \big\|{1\over 2}\big\{\nabla v_\delta^T\nabla v_\delta-\GI_3\big\}\big\|_{L^2({\cal S}_\delta ; \R^{3\times 3})}\le  C\delta^2,\\
 &||\nabla v_\delta||_{L^4({\cal S}_\delta ; \R^{3\times 3})}\le  C\delta^{1\over 2}.
 \end{aligned}
 \end{equation} For any fixed $\delta$, the deformation $v_\delta$ is decomposed as in Theorem \ref{The3.3}.  There exists a subsequence still indexed by $\delta$ such that  (see Subsection \ref{Subcase2})
 \begin{equation}\begin{aligned}\label{8weakconv2}
&\GR_\delta \rightharpoonup \GR^{\{0\}} \quad  \hbox{weakly in}\quad H^1({\cal S} ; SO(3)),\\
&{\cal V}_\delta\longrightarrow  {\cal V}^{\{0\}} \quad \hbox{strongly in}\quad  H^1({\cal S} ; \R^3),\\
&{1\over \delta}\big({d{\cal V}_{i,\delta}\over ds_i}-\GR_{i,\delta}\Gt_i\big)\rightharpoonup   {\cal Z}^{\{0\}}_i \quad \hbox{weakly  in}\quad  L^2(0,L_i ; \R^3),\\
&{1\over \delta^2}\Pi_{i,\delta}\big(\overline {v}_{\delta}\big) \rightharpoonup  \overline{v}^{\{0\}}_i \quad \hbox{weakly in}\quad  L^2(0,L_i;H^1(\omega ; \R^3)).
\end{aligned}
\end{equation}
The couple $({\cal V}^{\{0\}},\GR^{\{0\}})$ belongs to $\V\R_2$.
\noindent Furthermore, we also have ($i\in\{1,\ldots,N\}$)
\begin{equation}\label{(9.7)}
\begin{aligned}
\Pi_{i,\delta}(v_\delta)&\longrightarrow     {\cal V}^{\{0\}}_i \quad \hbox{strongly in}\quad H^1(\Omega_i ; \R^3),\\
\Pi_{i,\delta} (\nabla_x v_\delta)& \rightharpoonup  \GR^{\{0\}}_i\quad \hbox{weakly in}\quad L^4(\Omega_i ; \R^{3\times 3}),\\
{1\over 2\delta}\Pi_{i,\delta} \big((\nabla_xv_\delta)^T\nabla_x v_\delta- & \GI_3\big)\rightharpoonup   \GE^{\{0\}}_i \quad\hbox{weakly in}\quad L^2(\Omega_i ; \R^{3\times 3}),
\end{aligned}
\end{equation}
 where the symmetric matrix $\GE^{\{0\}}_i=(\Gt_i\,|\, \Gn_i\,|\, \Gb_i)\, \widehat{\GE}^{\{0\}}_i\,(\Gt_i\,|\, \Gn_i\,|\, \Gb_i)^T$,  (see Subsection \ref{Subcase2}).
The matrix   $\GE^{\{0\}}_i$  is defined by  
\begin{equation*}\label{HatEcase2}
\widehat{\GE}^{\{0\}}_i= \begin{pmatrix}
\ds-Y_2\Gamma_{i,3}(\GR^{\{0\}})+Y_3\Gamma_{i,2}(\GR^{\{0\}})+{\cal Z}^{\{0\}}_i\cdot\GR^{\{0\}}_i\Gt_i & * & * \\ \\
\ds -{1\over 2}Y_3\Gamma_{i,1}(\GR^{\{0\}})+{1\over 2}{\partial\overline{u}^{\{0\}}_i\over \partial Y_2}\cdot\Gt_i & \ds e_{22}(\overline{u}^{\{0\}}_i) & * \\ \\
\ds {1\over 2}Y_2\Gamma_{i,1}(\GR^{\{0\}})+{1\over 2}{\partial\overline{u}^{\{0\}}_i\over \partial Y_3}\cdot\Gt_i &  \ds e_{23}(\overline{u}^{\{0\}}_i) &\ds e_{33}(\overline{u}^{\{0\}}_i) \end{pmatrix}
 \end{equation*} with
 \begin{equation*}\label{ubarcase2}
\overline{u}^{\{0\}}_i=\big[Y_2\big({\cal Z}^{\{0\}}_i\cdot\GR^{\{0\}}_i\Gn_i\big)+Y_3\big({\cal Z}^{\{0\}}_i\cdot\GR^{\{0\}}_i\Gb_i\big)\big]\Gt_i+(\GR^{\{0\}}_i)^T\overline{v}^{\{0\}}_i,\quad i\in\{1,\ldots,N\}.\end{equation*}
Due to the decomposition of $v_\delta$, the assumptions \eqref{ForceN}, \eqref{ForceP} and  the above convergences \eqref{8weakconv2} 
we immediately get the limit of the term involving the body forces
\begin{equation*}
 \lim_{\delta\to 0}{1\over \delta^4}\int_{{\cal S}_\delta}f_{2,\delta}\cdot(v_\delta -I_d)= {\cal L}({\cal V}^{\{0\}},\GR^{\{0\}})
  \end{equation*} where ${\cal L}({\cal V},\GR)$ is defined by \eqref{L1}.
 
 We  now consider  a given rod ${\cal P}_{i,\delta}$.  Let $\varepsilon>0$ be fixed. Due to  assumption \eqref{HypW}, there exists $\theta>0$ such that 
\begin{equation}\label{Weps1}
\forall E\in\GS_3,  \enskip   |||E|||\le \theta, \enskipÊ W(E)\ge Q(E)-\varepsilon |||E|||^2.
\end{equation}
Let us denote by $\chi_{i,\delta}^\theta$ the characteristic function of the set
 $$A_{i,\delta}^\theta=\{s\in\Omega_i\setminus\bigcup_{k=1}^{K_i}\omega\times[a_i^k-\rho_0\delta,a_i^k+\rho_0\delta] \;;\; |||\Pi_{i,\delta} \big((\nabla_xv_\delta)^T\nabla_x v_\delta-\GI_3\big)(s)||| \geq \theta\}$$ where $a_i^k$ is the arc length of a knot belonging to the line $\gamma_i$.  Due to \eqref{(9.7)}, we have
\begin{equation}\label{mesA1} {\rm meas}(A_{i,\delta}^\theta)\le C{\delta^2\over \theta^2}.
\end{equation}
Using the positive character of $W$, \eqref{W}, \eqref{Weps1} and \eqref{GSV2} give
\begin{equation*}
\begin{aligned}
{1\over \delta^4} \int_{{\cal S}_\delta}\widehat{W}\big(\nabla_xv_\delta\big) &\ge \sum_{i=1}^N \int_{\Omega_i}{1\over \delta^2} W\Big({1\over 2}\Pi_{i,\delta} \big((\nabla_xv_\delta)^T\nabla_x v_\delta-\GI_3\big)\Big)(1-\chi_{i,\delta}^\theta)\\
&\ge  \sum_{i=1}^N  \int_{\Omega_i} Q\Big({1\over 2\delta}\Pi_{i,\delta} \big((\nabla_xv_\delta)^T\nabla_x v_\delta-\GI_3\big)(1-\chi_{i,\delta}^\theta)\Big)-C\varepsilon.
\end{aligned}
\end{equation*}
In view of \eqref{mesA1}, the function $\chi_{i,\delta}^\theta$ converges  a.e. to $0$ as $\delta$ tends to $0$ while the weak limit of $\ds{1\over 2\delta}\Pi_{i,\delta} \big((\nabla_xv_\delta)^T\nabla_x v_\delta-\GI_3\big)(1-\chi_{i,\delta}^\theta)$ is given by \eqref{(9.7)}. As a consequence  we have 
$$\liminf_{\delta\to 0}{1\over \delta^4} \int_{{\cal S}_\delta}\widehat{W}\big(\nabla_xv_\delta\big)\ge  \sum_{i=1}^N 
\int_{\Omega_i} Q\big(\GE^{\{0\}}_i\big)-C\varepsilon.$$
As $\varepsilon$ is arbitrary, this gives
\begin{equation}\label{Wlimite}
\liminf_{\delta\to 0} {1\over \delta^4} \int_{{\cal S}_\delta}\widehat{W}\big(\nabla_xv_\delta\big)\ge  \sum_{i=1}^N \int_{\Omega_i} Q\big(\GE^{\{0\}}_i\big).
\end{equation}
Hence, due to the limit of the applied forces and \eqref{Wlimite} we obtain
\begin{equation*}
\begin{aligned}
\sum_{i=1}^N \int_{\Omega_i} Q\big(\GE^{\{0\}}_i\big)-{\cal L}({\cal V}^{\{0\}},\GR^{\{0\}})\le \liminf_{\delta\to 0}{m_{2,\delta}\over \delta^{4}}
 \end{aligned}
 \end{equation*}  where ${\cal L}({\cal V},\GR)$ is defined  by \eqref{L1}.

\noindent Finally, for a.e. $s_i\in ]0,L_i[$ we minimize the quantity $\ds \int_\omega Q\big(\GE^{\{0\}}_i\big)(s_i, Y_2, Y_3)dY_2dY_3$ with respect to the ${\cal Z}^{\{0\}}_i (s_i)\cdot\GR^{\{0\}}_i (s_i)\Gt_i (s_i)$'s and the $\overline{u}^{\{0\}}_i(s_i, ., .)$'s using Lemma \ref{correctors} in the Appendix. Hence we obtain  
 \begin{equation}\label{finalstep1k=2}
 \begin{aligned}
\min_{({\cal V},\GR)\in \V\R_2} {\cal J}_2\big({\cal V},\GR\big)&\le {\cal J}_2\big({\cal V}^{\{0\}},\GR^{\{0\}}\big)\\
&=\sum_{i=1}^N \int_0^{L_i} a_{lk}\Gamma_{i,k}(\GR^{\{0\}})\Gamma_{i,l}(\GR^{\{0\}})-{\cal L}({\cal V}^{\{0\}},\GR^{\{0\}})\\
&\le \sum_{i=1}^N \int_{\Omega_i} Q\big(\GE^{\{0\}}_i\big)-{\cal L}({\cal V}^{\{0\}},\GR^{\{0\}})\le \liminf_{\delta\to 0}{m_{2,\delta}\over \delta^{4}},
\end{aligned}
 \end{equation}
where the  definite positive symmetric $3\times 3$ matrix $\GA=(a_{lk})$ is defined in Lemma \ref{correctors}.

\noindent {\it Step 2.} In this step we show that 
$$\limsup_{\delta\to 0}{m_{2,\delta}\over \delta^{4}}\le \min_{({\cal V},\GR)\in \V\R_2} {\cal J}_2\big({\cal V},\GR\big).$$
Let $({\cal V}^{\{2\}},\GR^{\{2\}})\in \V\R_2$  be such that $\ds{\cal J}_2\big({\cal V}^{\{2\}},\GR^{\{2\}}\big)=\min_{({\cal V},\GR)\in \V\R_2} {\cal J}_2\big({\cal V},\GR\big)$. 
First, for $i\in\{1,\ldots,N\}$ let   $\overline {v}_i$ be arbitrary in $W^{1,\infty}(\Omega_i;\R^3)$.

\noindent For each rod ${\cal P}_{i, \delta}$, we apply the  Proposition \ref{PropAp} given in Appendix to the triplet $({\cal V}^{\{2\}}_i,\GR^{\{2\}}_i, \overline{v}_i)$ in each portion of  ${\cal P}_{i, \delta}$ which is contained between an extremity and a knot or between two knots.   Doing such leads to a sequence of deformations $v_\delta$ which belong to $\D_\delta\cap W^{1, \infty}({\cal S}_\delta;\R^3)$ and which satisfy in each junction ${\cal J}_{A,\rho_0\delta}$
\begin{equation}\label{vtestJ2}
\forall A\in {\cal K},\qquad v_\delta(x)={\cal V}^{\{2\}}(A)+\GR^{\{2\}}(A)(x-A),\qquad x\in {\cal J}_{A,\rho_0\delta},\\
\end{equation} and 
\begin{equation}\label{vtestC2}
\begin{aligned}
&\det\big(\nabla  v_\delta(x)\big)>0 \qquad \hbox{for a.e. } x\in {\cal S}_\delta,\\
& \Pi_{i,\delta}(v_\delta)\longrightarrow {\cal V}^{\{2\}}_i\qquad \hbox{strongly in }\enskip H^1(\Omega_i ; \R^3),\\
&{1\over \delta} \Pi_{i,\delta}(v_\delta-{\cal V}_{i,\delta})\longrightarrow \GR^{\{2\}}_i\big(Y_2\Gn_i+Y_3\Gb_i)\qquad \hbox{strongly in }\enskip L^2(\Omega_i ; \R^3),\\
&{1\over 2\delta}\Pi_{i,\delta}\big((\nabla v_\delta)^T\nabla v_\delta
-\GI_3\big)\longrightarrow \GE^{\{2\}}_i=\big({d\GR^{\{2\}}_i\over ds_i}\big(Y_2\Ge_2+Y_3\Ge_3)\;|\; {\partial \overline{v}_i\over \partial Y_2} \; |\; {\partial \overline{v}_i\over \partial Y_3} \;\big)^T\GR^{\{2\}}_i\\
&\qquad +(\GR^{\{2\}}_i)^T\big({d\GR^{\{2\}}_i\over ds_i}\big(Y_2\Ge_2+Y_3\Ge_3)\;|\; {\partial \overline{v}_i\over \partial Y_2} \; |\; {\partial \overline{v}_i\over \partial Y_3}\;\big)\quad\hbox{strongly in }\enskip L^2(\Omega_i ; \R^{3\times 3})
\end{aligned}
\end{equation} where ${\cal V}_{i,\delta}$ is the average of $v_\delta$ on each cross-section of the rod ${\cal P}_{i,\delta}$. 
 \noindent Moreover there exists a constant $C_1\ge \theta$ which does not depend on $\delta$ such that
\begin{equation}\label{EstLinf}
\forall i\in\{1,\ldots,N\},\qquad ||\Pi_{i,\delta}\big((\nabla v_\delta)^T\nabla v_\delta
-\GI_3\big)||_{L^\infty(\Omega_i ; \R^{3\times 3})}\le C_1.
\end{equation}
\vskip 1mm
Let $\varepsilon>0$ be fixed. Due to  assumption \eqref{HypW}, there exists $\theta>0$ such that 
\begin{equation}\label{Weps}
\forall E\in\GS_3,  \enskip   |||E|||\le \theta, \enskipÊ W(E)\le Q(E)+\varepsilon |||E|||^2.
\end{equation}
Let us denote by $\chi_{i,\delta}^\theta$ the characteristic function of the set
 $$A_{i,\delta}^\theta=\{s\in\Omega_i\setminus\bigcup_{k=1}^{K_i}\omega\times[a_i^k-\rho_0\delta,a_i^k+\rho_0\delta] \;;\; |||\Pi_{i,\delta} \big((\nabla_xv_\delta)^T\nabla_x v_\delta-\GI_3\big)(s)||| \geq \theta\}$$ where $a_i^k$ is the arc length of a knot belonging to the line $\gamma_i$.  Due to \eqref{(9.7)}, we have
\begin{equation}\label{mesA} {\rm meas}(A_{i,\delta}^\theta)\le C{\delta^2\over \theta^2}.
\end{equation}
Using the fact that $v_\delta$ is equal to a rotation in the junctions (see \eqref{vtestJ2}) the Saint Venant's strain tensor is equal to zero in the junctions. Hence we have
\begin{equation}\label{Decoup}
\begin{aligned}
{1\over \delta^4} \int_{{\cal S}_\delta}\widehat{W}\big(\nabla_xv_\delta\big) &= \sum_{i=1}^N \int_{\Omega_i}{1\over \delta^2} W\Big({1\over 2}\Pi_{i,\delta} \big((\nabla_xv_\delta)^T\nabla_x v_\delta-\GI_3\big)\Big)(1-\chi_{i,\delta}^\theta)\\
&+ \sum_{i=1}^N \int_{\Omega_i}{1\over \delta^2} W\Big({1\over 2}\Pi_{i,\delta} \big((\nabla_xv_\delta)^T\nabla_x v_\delta-\GI_3\big)\Big)\chi_{i,\delta}^\theta.
\end{aligned}
\end{equation} In view of  \eqref{Weps} and the third strong convergence in \eqref{vtestC2}, the first term of the right hand side is estimated as 
\begin{equation*}
\begin{aligned}
&\sum_{i=1}^N \int_{\Omega_i}{1\over \delta^2} W\Big({1\over 2}\Pi_{i,\delta} \big((\nabla_xv_\delta)^T\nabla_x v_\delta-\GI_3\big)\Big)(1-\chi_{i,\delta}^\theta)\\
&\le  \sum_{i=1}^N  \int_{\Omega_i} Q\Big({1\over 2\delta}\Pi_{i,\delta} \big((\nabla_xv_\delta)^T\nabla_x v_\delta-\GI_3\big)(1-\chi_{i,\delta}^\theta)\Big)+C\varepsilon.
\end{aligned}
\end{equation*}
Again, the third  strong convergence in \eqref{vtestC2}  and  the pointwise convergence of the function  $\chi_{i,\delta}^\theta$ allows to pass to the limit as $\delta$ in the above inequality and to obtain 
$$\limsup_{\delta\to 0}\sum_{i=1}^N \int_{\Omega_i}{1\over \delta^2} W\Big({1\over 2}\Pi_{i,\delta} \big((\nabla_xv_\delta)^T\nabla_x v_\delta-\GI_3\big)\Big)(1-\chi_{i,\delta}^\theta)\le  \sum_{i=1}^N \int_{\Omega_i} Q\big(\GE^{\{2\}}_i\big)+C\varepsilon.$$

Let us recall estimate \eqref{EstLinf}.  Due to the continuity of $W$ and the third  assumption of \eqref{HypW}, there exists a constant $C_2$ such that
$$|||E|||\le C_1 \qquad \Longrightarrow \qquad W(E)\le C_2|||E|||^2.$$
It follows that the second term in \eqref{Decoup} is less than 
\begin{equation*}
\begin{aligned}
\sum_{i=1}^N \int_{\Omega_i}{1\over \delta^2} W\Big({1\over 2}\Pi_{i,\delta} \big((\nabla_xv_\delta)^T\nabla_x v_\delta-\GI_3\big)\Big)\chi_{i,\delta}^\theta\le \sum_{i=1}^N \int_{\Omega_i} C_2|||{1\over 2\delta}\Pi_{i,\delta} \big((\nabla_xv_\delta)^T\nabla_x v_\delta-\GI_3\big)|||^2 \chi_{i,\delta}^\theta.
\end{aligned}
\end{equation*} 
We have $\ds \chi_{i,\delta}^\theta$ tends to 0 weakly *  in $L^\infty(\Omega_i)$ and the third strong convergence in \eqref{vtestC2}, hence
\begin{equation*}
\begin{aligned}
\lim_{\delta\to 0}\sum_{i=1}^N \int_{\Omega_i}{1\over \delta^2} W\Big({1\over 2}\Pi_{i,\delta} \big((\nabla_xv_\delta)^T\nabla_x v_\delta-\GI_3\big)\Big)\chi_{i,\delta}^\theta=0.
\end{aligned}
\end{equation*} As $\varepsilon$ is arbitrary, finally we get
$$\limsup_{\delta\to 0}{1\over \delta^4} \int_{{\cal S}_\delta}\widehat{W}\big(\nabla_xv_\delta\big)\le \sum_{i=1}^N \int_{\Omega_i} Q\big(\GE^{\{2\}}_i\big).$$
Thanks to the first and second convergences in \eqref{vtestC2} we obtain the limit of the term involving the forces. Then we obtain
\begin{equation}
\begin{aligned}
\limsup_{\delta\to 0}{m_{2,\delta}\over \delta^{4}}\le\limsup_{\delta \to 0}{J_{2,\delta}(v_\delta)\over \delta^4}\le  \sum_{i=1}^N \int_{\Omega_i} Q\big(\GE^{\{2\}}_i\big)-{\cal L}({\cal V}^{\{2\}},\GR^{\{2\}})
 \end{aligned}
 \end{equation}  where ${\cal L}({\cal V},\GR)$ is defined in \eqref{L1}.
In view of the expression of the $\GE^{\{2\}}_i$'s and a density argument the above inequality holds true for any  family  $\overline {v}_i\in L^2(0,L_i;H^1(\omega;\R^3))$ ($i\in\{1,\ldots,N\}$). Then we can use again Lemma \ref{correctors} in order to minimize the right hand side of this inequality with respest to the $(\GR^{\{2\}})^T\overline{v}_i(s_i, ., .)$'s. This gives
 \begin{equation*}
 \begin{aligned}
\limsup_{\delta\to 0}{m_{2,\delta}\over \delta^{4}}&\le \min_{({\cal V},\GR)\in \V\R_2} {\cal J}_2\big({\cal V},\GR\big)= {\cal J}_2\big({\cal V}^{\{2\}},\GR^{\{2\}}\big)\\
&=\sum_{i=1}^N \int_0^{L_i} a_{lk}\Gamma_{i,k}(\GR^{\{2\}})\Gamma_{i,l}(\GR^{\{2\}})-{\cal L}({\cal V}^{\{2\}},\GR^{\{2\}})\\
&\le \sum_{i=1}^N \int_{\Omega_i} Q\big(\GE^{\{2\}}_i\big)-{\cal L}({\cal V}^{\{2\}},\GR^{\{2\}}).
\end{aligned}
 \end{equation*} This conclude the proof of the theorem. 
 \end{proof}

\begin{remark} Let us point out that Theorem \ref{theo9.1} shows that for any minimizing sequence $(v_\delta)_\delta$ as in Step 1, the third convergence of the rescaled Green-St Venant's strain tensor in \eqref{(9.7)} is a strong convergence to $\GE_i^{\{0\}}$ in $L^2(\Omega_i;\R^{3\times 3})$.
\end{remark}
\section{Appendix}\label{appendix}
 
In this Appendix, we first give the construction of a suitable sequence (Proposition \ref{PropAprox}) of deformations to prove the second step of the proof  of Theorem \ref{theo8.1}. To do that we first give three lemmas. In a second part of this appendix we also construct  a suitable sequence (Proposition \ref{PropAp}) of deformations to prove the second step of the proof  of Theorem \eqref{theo9.1}. At last Lemma \ref{correctors} provides the elimination technique used  in the proof  of Theorem \eqref{theo9.1}.

The proofs of the three lemmas below are left to the reader.

\begin{lemma}\label{LEMap1} Let $ \GR$ be an element in  $ L^2(0,L ; {\cal C})$. We define  the field $\GR^{'}_n $ of   $L^2(0,L ; {\cal C})$ by
\begin{equation*}
\GR^{'}_n(t)={n\over L}\int_{kL/n}^{(k+1)L/n}\GR(s)ds\quad\hbox{for any $t$ in} \quad ]kL/n,(k+1)L/n[,\quad k\in\{0,\ldots, n-1\}\\
\end{equation*}
We have  
$$ \GR^{'}_n\longrightarrow \GR \enskip\text{strongly in } L^2(0, L;{\cal C}).$$
\end{lemma} 
\begin{lemma}\label{LEMap1bis} Let $\GR$ be in $ {\cal C}$. There exist $(\lambda_1,\ldots,\lambda_p)\in [0,1]^p$ and $(\GR_1,\ldots , \GR_p)\in \big(SO(3))^p$ such that
 $$\sum_{i=1}^p \lambda_i=1,\qquad \GR=\sum_{i=1}^p \lambda_i\GR_i.$$ We set
 \begin{equation*}
 \mu_0={1\over 2n},\qquad
 \mu_i= \mu_{i-1}+\Big(1-{1\over n}\Big)\lambda_{i},\qquad i\in\{1,\ldots, p\}. 
 \end{equation*}
 We define  $\GR^{'}_n $ in  $L^2(0,L ; SO(3))$ by: for any $k\in\{0,\ldots, n-1\}$ and for a.e.  $y\in ]0, 1[$ we set
\begin{equation*}
\GR^{'}_n\big({k\over n}+{y\over n} \big)=
\left\{\begin{aligned}
&\GI_3\qquad y\in ]0,\mu_0[,\\
&\GR_i\qquad y\in]\mu_{i-1},\mu_{i}[  \quad, i\in \{1,\ldots, p\},\\
&\GI_3\qquad y\in ]1-1/2n, 1[.
\end{aligned}\right.
\end{equation*}
We have 
$$\GR^{'}_n\rightharpoonup \GR\quad \hbox{weakly in }\enskip L^2(0, L;{\cal C}).$$
\end{lemma}

\begin{lemma}\label{LEMap4} Let $\theta^*\in [0,\pi]$ there exists a function $\theta\in W^{1,\infty}(0,1)$ such that
\begin{equation*}
\theta(0)=0,\quad \theta(1)=\theta^*,\qquad \int_0^1e^{i\theta(t)}dt={1+e^{i\theta^*}\over 2}.
\end{equation*} Moreover there exists a positive constant which does not depend on $\theta^*$ such that
$$||\theta^{'}||_{L^\infty(0,1)}\le C\theta^*.$$
\end{lemma}
The above lemmas allow us to establish the following result.

 \begin{prop}\label{PropAprox} Let $({\cal V},\GR)$ be in  $\V\R_1$. There exists  a sequence  $\big(({\cal V}_n,\GR_n)\big)_n$ in $ \V\R_1$ which satisfies $\GR_n\in W^{1,\infty}({\cal S};SO(3))$, which is equal to $\GI_3$ in a neighbourhood  of each knot $A\in {\cal K}$ and each  fixed extremity belonging to $\Gamma_0$ and moreover which satisfies
\begin{equation*}
\begin{aligned}
{\cal V}_n \rightharpoonup {\cal V} \quad \hbox{weakly  in} \;H^1({\cal S};\R^3),\\
\GR_n\rightharpoonup \GR \quad \hbox{weakly in} \; L^2({\cal S};{\cal C}).
\end{aligned}
\end{equation*} 
 \end{prop} 
 \begin{proof} For any $n\in \N^*$, we first apply the Lemma \ref{LEMap1} between two consecutive points of the set ${\cal K}\cup \Gamma$ belonging to a same  line $\gamma_i$ of the skeleton ${\cal S}$.  It leads to a sequence $\GR^{'}_n\in L^2({\cal S}; {\cal C})$ which is piecewise constant on ${\cal S}$. Then, we define ${\cal V}^{'}_n$ by integration, using the formula $\ds {d({\cal V}^{'}_n)_i\over ds_i}= (\GR^{'}_n)_i\Gt_i$ between two consecutive points in ${\cal K}\cup \Gamma$ imposing the values 
 $${\cal V}^{'}_n(A)={\cal V} (A)\;\; \hbox{for every point in } {\cal K}\cup\Gamma,$$ which is possible because the means of  $(\GR^{'}_n)_i$ are preserved between two points in ${\cal K}\cup \Gamma$.
 
Hence, we obtain  a sequence  $\big(({\cal V}^{'}_n,\GR^{'}_n)\big)_n$ in $\V\R_1$  which satisfies 
 \begin{equation*}
\begin{aligned}
&{\cal V}^{'}_n \longrightarrow {\cal V} \quad \hbox{strongly  in} \;H^1({\cal S};\R^3),\\
&\GR^{'}_n \longrightarrow \GR \quad \hbox{strongly in} \; L^2({\cal S};{\cal C}).
\end{aligned}
\end{equation*}Then, we consider the couple $\Big(\big(1-{1/n}\big){\cal V}^{'}_n+{1/ n}\phi,\big(1-{1/n}\big)\GR^{'}_n+{1/ n}\GI_3\big)\Big)\in\V\R_1$ and we apply the Lemma \ref{LEMap1bis}, again  between two consecutive points of ${\cal K}\cup \Gamma$ on a same  line $\gamma_i$ of the skeleton ${\cal S}$. We obtain  an element $({\cal V}^{''}_n,\GR^{''}_n)\in\V\R_1$ such that $\GR^{''}_n$ belongs to $L^2({\cal S}; SO(3))$ and is piecewise constant on ${\cal S}$ and equal to $\GI_3$ in a neighbourhood of each knot and each  fixed extremity belonging to $\Gamma_0$ and moreover which satisfies 
 \begin{equation*}
\begin{aligned}
&{\cal V}^{''}_n \rightharpoonup {\cal V} \quad \hbox{weakly  in} \;H^1({\cal S};\R^3),\\
&\GR^{''}_n\rightharpoonup \GR \quad \hbox{weakly in} \; L^2({\cal S};{\cal C}).
\end{aligned}
\end{equation*} Finally, thanks to the Lemma \ref{LEMap4}, we replace the element  $({\cal V}^{''}_n,\GR^{''}_n)$ by another  $({\cal V}^{'''}_n,\GR^{'''}_n)$ in $\V\R_1$ such that $\GR^{'''}_n$ belongs to $W^{1,\infty}({\cal S}; SO(3))$. Notice that we do this modification without changing the values of  ${\cal V}^{''}_n$ at the points of ${\cal K}\cup \Gamma$. We get
 \begin{equation*}
\begin{aligned}
&{\cal V}^{'''}_n \rightharpoonup {\cal V} \quad \hbox{weakly  in} \;H^1({\cal S};\R^3),\\
&\GR^{'''}_n\rightharpoonup \GR \quad \hbox{weakly in} \; L^2({\cal S};{\cal C}).
\end{aligned}
\end{equation*} The proposition is proved.
\end{proof}
 The proposition below gives an approximation result of a deformation in a rod by deformations which are rotations near the extremities. Recall that for a single rod $\Omega_\delta=]0,L[\times \omega_\delta$, the rescaling into $\Omega=]0,L[\times \omega$ is performed through the operator $\Pi_\delta $ defined by  
$$\Pi_{\delta} (\psi)(y_1,Y_2,Y_3)=\psi(s,\delta Y_2,\delta Y_3)\quad\hbox{ for a.e.}\quad (y_1,Y_2,Y_3) \in \Omega.$$
\begin{prop}\label{PropAp} Let $\big({\cal V} , \GR, \overline{v}\big)$ be in $H^1(0, L ; \R^3)\times H^1(0, L ; SO(3))\times  W^{1,\infty}(\Omega ; \R^3)$  such that $\ds{d{\cal V}\over dy_1}=\GR\Ge_1$. Let $\rho>0$ be given. There exists a sequence of  deformations $v_\delta
\in W^{1,\infty}(\Omega_\delta ; \R^3)$  satisfying for $\delta$ small enough 
\begin{equation}\label{vtest}
\begin{aligned}
v_\delta(y)&={\cal V}(0)+\GR(0) y,\qquad y\in]0,\rho\delta
[\times \omega_\delta,\\
v_\delta(y)&={\cal V}(L)+\GR(L)(y-L\Ge_1),\qquad y\in]L-\rho\delta, L[\times \omega_\delta,\\
&\det\big(\nabla v_\delta(y)\big)>0 \quad \hbox{for almost any } y\in \Omega_\delta,\\
& ||\Pi_{\delta}\big((\nabla v_\delta)^T\nabla v_\delta
-\GI_3\big)||_{L^\infty(\Omega ; \R^{3\times 3})}\le C,\\
\end{aligned}
\end{equation}
where the constant $C$ does not depend on $\delta$.

Moreover the following convergence holds true
\begin{equation}\label{vtest2}
\begin{aligned}
{1\over 2\delta}\Pi_\delta\big((\nabla v_\delta)^T&\nabla v_\delta
-\GI_3\big)\longrightarrow \big({d\GR\over dy_1}\big(Y_2\Ge_2+Y_3\Ge_3)\;|\; {\partial \overline{v}\over \partial Y_2} \; |\; {\partial \overline{v}\over \partial Y_3} \;\big)^T\GR\\
+\GR^T\big({d\GR\over dy_1}&\big(Y_2\Ge_2+Y_3\Ge_3)\;|\; {\partial \overline{v}\over \partial Y_2} \; |\; {\partial \overline{v}\over \partial Y_3}\;\big)\quad\hbox{strongly in }\enskip L^2(\Omega ; \R^{3\times 3}).
\end{aligned}
\end{equation}
\end{prop}
\begin{proof} Let $N\ge 4$ and $\eps=L/N$. At the end of the proof, we will fix  $\eps\ge \delta$ in terms of $\delta$ and $\rho$ or equivalently $N$ in terms of $\delta$ and $\rho$. We define $\GR_\eps\in W^{1,\infty}(0,L ; \R^{3\times 3})$ as follows:
\begin{equation*}
\begin{aligned}
&\GR_\eps\quad\hbox{is piecewise linear on } [0,L],\\
&\GR_\eps(y_1)=\GR(0)\qquad y_1\in[0,\eps],\\
&\GR_\eps(k\eps)=\GR(k\eps),\quad k\in\{2,\ldots, N-2\}\\
&\GR_\eps(y_1)=\GR(L)\qquad y_1\in[L-\eps, L],\\
\end{aligned}
\end{equation*} There exists a constant $C$ independent of $\eps$ such that 
\begin{equation*}
\begin{aligned}
&\Big\|{d\GR_\eps\over dy_1}\Big\|_{L^2(0, L ; \R^{3\times 3})}\le C\Big\|{d\GR\over dy_1}\Big\|_{L^2(0, L ; \R^{3\times 3})},\qquad \Big\|{d\GR_\eps\over dy_1}\Big\|_{L^\infty(0, L ; \R^{3\times 3})}\le {C\over \eps},\\
& \big\|\GR_\eps - \GR\big\|_{L^2(0, L ; \R^{3\times 3})}\le C\eps\Big\|{d\GR\over dy_1}\Big\|_{L^2(0, L ; \R^{3\times 3})},\quad   \big\|\GR_\eps - \GR\big\|_{L^\infty(0, L ; \R^{3\times 3})}\le C\sqrt\eps\Big\|{d\GR\over dy_1}\Big\|_{L^2(0, L ; \R^{3\times 3})}.\\
\end{aligned}
\end{equation*}
Due to the specific construction of the sequence  $\GR_\varepsilon$, one also get 
\begin{equation}\label{10.conv}
\begin{aligned}
\GR_\eps\longrightarrow \GR\quad \hbox{strongly in }\quad  H^1(0, L ; \R^{3\times 3}),\\
{1\over \eps}\big(\GR_\eps-\GR)\longrightarrow 0\quad \hbox{strongly in }\quad  L^2(0, L ; \R^{3\times 3}).
\end{aligned}
\end{equation}
 Now we define ${\cal V}_\eps$ by
\begin{equation*}
\begin{aligned}
&{\cal V}_\eps\quad\hbox{is piecewise linear on } [0,L],\\
&{\cal V}_\eps(y_1)={\cal V}(0)+y_1\GR(0)\Ge_1\qquad y_1\in[0,\eps],\\
&{\cal V}_\eps(k\eps)={\cal V}(k\eps),\quad k\in\{2,\ldots, N-2\},\\
&{\cal V}_\eps(y_1)={\cal V}(L)+(y_1-L)\GR(L)\Ge_1\qquad y_1\in[L-\eps, L].
\end{aligned}
\end{equation*} We have
\begin{equation*}
\begin{aligned}
&\Big\|{d{\cal V}_\eps\over dy_1}-{d{\cal V}\over dy_1}\Big\|_{L^2(0, L ; \R^{3})}\le C\eps\Big\|{d\GR\over dy_1}\Big\|_{L^2(0, L ; \R^{3\times 3})},\\
&\Big\|{d{\cal V}_\eps\over dy_1}-{d{\cal V}\over dy_1}\Big\|_{L^\infty(0, L ; \R^{3})}\le  C\sqrt\eps\Big\|{d\GR\over dy_1}\Big\|_{L^2(0, L ; \R^{3\times 3})},
\end{aligned}
\end{equation*}
and the following convergence
\begin{equation}\label{11.conv}
{1\over \eps}\Big({d{\cal V}_\eps\over dy_1}-{d{\cal V}\over dy_1}\Big)\longrightarrow 0\quad \hbox{strongly in }\quad  L^2(0, L ; \R^{3}).
\end{equation}
We set
$$d_\eps(y_1)=\inf\Big(1, \sup\big(0, {y_1-\eps\over \eps}\big), \sup\big(0, {L-\eps-y_1\over \eps}\big)\Big), \qquad y_1\in [0,L].$$

Now, we consider the deformation $v_\eps$ defined by
$$v_\eps(y)={\cal V}_\eps(y_1)+\GR_\eps(y_1)\big(y_2\Ge_2+y_3\Ge_3)+\delta^2d_\eps(y_1)\overline{v}\big(y_1, {y_2\over \delta}, {y_3\over \delta}\big),\qquad x\in \Omega_\delta.$$
A straightforward calculation gives
\begin{equation*}
\begin{aligned}
(\nabla v_\eps)^T\nabla v_\eps-\GI_3&=(\nabla v_\eps-\GR_\eps)^T\GR_\eps+\GR_\eps^T(\nabla v_\eps-\GR_\eps)+(\nabla v_\eps-\GR_\eps)^T(\nabla v_\eps-\GR_\eps)\\
&+(\GR_\eps-\GR)^T\GR_\eps+\GR^T(\GR_\eps-\GR),\\
{\partial  v_\eps\over \partial y_1}-\GR_\eps\Ge_1& ={d{\cal V}_\eps\over dy_1}- \GR_\eps\Ge_1+{d\GR_\eps\over dy_1}\big(y_2\Ge_2+y_3\Ge_3)+\delta^2\Big({\partial\overline{v}\over \partial y_1}\big(\cdot, {y_2\over \delta}, {y_3\over \delta}\big)+{d d_\eps\over dy_1}\overline{v}\big(\cdot, {y_2\over \delta}, {y_3\over \delta}\big)\Big),\\
{\partial  v_\eps\over \partial y_\alpha}-\GR_\eps\Ge_\alpha& = \delta d_\eps{\partial\overline{v}\over \partial Y_\alpha}\big(\cdot, {y_2\over \delta}, {y_3\over \delta}\big),\qquad \alpha\in\{2,3\}.
\end{aligned}
\end{equation*} Hence, thanks to the above estimates we obtain
$$||\nabla v_\eps-\GR||_{L^\infty(\Omega_\delta ; \R^{3\times 3})}\le C\Big(\sqrt\eps\Big\|{d\GR\over dy_1}\Big\|_{L^2(0, L ; \R^{3\times 3})}+{\delta\over \eps}+{\delta^2\over \eps}||\overline{v}||_{W^{1,\infty}(\Omega ; \R^3\Big)}.$$ Finally, we choose $\theta \ge \rho$ and large enough so that $\theta \ge 4C$. Then setting $\eps=\theta\delta$ show that if $\delta$ is small enough (denoting now $v_\eps$ by $v_\delta$)
$$||\nabla v_\delta-\GR||_{L^\infty(\Omega_\delta ; \R^{3\times 3})}\le {1\over 2},$$
 which shows first  that  the determinant of $\nabla v_\eps$ is positive.  Moreover the above inequality also implies (e.g.) that $||(\nabla v_\delta)^T\nabla v_\delta-\GI_3||_{L^\infty(\Omega_\delta ; \R^{3\times 3})}\le 2$.

 At least the above decomposition of $(\nabla v_\delta)^T\nabla v_\delta-\GI_3$, the strong convergences \eqref{10.conv} and  \eqref{11.conv} together with the definition of $d_\varepsilon$ allows to obtain the strong convergence \eqref{vtest2}.
\end{proof} 
We know justify the elimination process used in Theorem  \ref{theo9.1}.
We set
$$\W=\Big\{\psi\in  H^1(\omega ; \R^3)\;|\; \int_\omega\psi(Y_2,Y_3)dY_2dY_3=0,\enskip\int_\omega\big(Y_3\psi_2(Y_2,Y_3)-Y_2\psi_3(Y_2,Y_3)\big)dY_2dY_3=0\Big\}$$ We equip $\W$ with the norm
$$||\psi||_\W=||\nabla\psi_1||_{L^2(\omega ; \R^2)}+||e_{22}(\psi)||_{L^2(\omega)}+||e_{23}(\psi)||_{L^2(\omega)}+||e_{33}(\psi)||_{L^2(\omega)}$$ where
$$e_{kl}(\psi)={1\over 2}\Big({\partial\psi_k\over \partial Y_l}+{\partial\psi_l\over \partial Y_k}\Big),\qquad (k,l)\in\{2,3\}^2.$$ In the space $\W$, the above norm is equivalent to the $H^1$ norm.
\vskip 1mm
Let $\GQ$ be a matrix field  belonging to $L^\infty(\omega ; \R^{6\times 6})$ satisfying for a. e. $ (Y_2,Y_3)\in \omega$
\begin{equation}\label{propQ}
\begin{aligned}
& \GQ(Y_2,Y_3)\hbox{ is a symmetric positive definite matrix}  ,\\ 
& \GQ(-Y_2, -Y_3)=\GQ(Y_2,Y_3),\\
& \forall \xi\in \R^6,\qquad 
c|\xi|^2\le \GQ(Y_2,Y_3)\, \xi \cdot \xi \le C|\xi|^2
\end{aligned}
\end{equation} where $c$ and $C$ are positive constants which do not depend on $(Y_2,Y_3)$. 

We denote  $\chi_i\in \W$, $i\in\{1,2,3\}$ the solutions of the following variational problems:
 \begin{equation}\label{chii}
 \begin{aligned}
& \forall \psi\in \W,\\
&\int_\omega  
\GQ\,
\begin{pmatrix}
0\\
\ds {1\over 2}{\partial \chi_{i,1}\over \partial Y_2}\\
\ds{1\over 2}{\partial \chi_{i,1}\over \partial Y_3}\\
e_{22}(\chi_i)\\  e_{23}(\chi_i)\\   e_{33}(\chi_i)
\end{pmatrix}\cdot
\begin{pmatrix}
0\\
\ds{1\over 2}{\partial \psi_1\over \partial Y_2}\\
\ds{1\over 2}{\partial \psi_1\over \partial Y_3}\\
e_{22}(\psi)\\  e_{23}(\psi)\\   e_{33}(\psi)
\end{pmatrix}=-
\int_\omega  
\GQ\,
\GV_i\cdot
\begin{pmatrix}
0\\
\ds {1\over 2}{\partial \psi_1\over \partial Y_2}\\
\ds{1\over 2}{\partial \psi_1\over \partial Y_3}\\
e_{22}(\psi)\\  e_{23}(\psi)\\   e_{33}(\psi)
\end{pmatrix}
\end{aligned}
\end{equation}
where
\begin{equation}\label{SCM}
\GV_1=\begin{pmatrix}
0\\
\ds -{1\over 2}Y_3\\
\ds{1\over 2}Y_2\\
0 \\  0 \\  0
\end{pmatrix},\quad
\GV_2=\begin{pmatrix}
Y_3\\
0\\
0\\
0 \\  0 \\  0
\end{pmatrix},\quad
\GV_3=\begin{pmatrix}
-Y_2\\
0\\
0\\
0 \\  0 \\  0
\end{pmatrix}.
\end{equation}
The fields $\chi_i$, $i\in\{1,2,3\}$, satisfy
\begin{equation}\label{propchi}
\chi_i(-Y_2,-Y_3)=\chi_i(Y_2,Y_3)\qquad \hbox{for a.e. } (Y_2,Y_3)\in \Omega.
\end{equation}

\begin{lemma}\label{correctors} There exists a symmetric positive definite matrix $\GA$ such that for any  $\Gamma_i\in \R$, $i\in\{1,2,3\}$, 
\begin{equation*}
\begin{aligned}
&\forall (Z,\psi)\in \R\times H^1(\omega;\R^3),\\
 &\int_\omega  
\GQ\,
\begin{pmatrix}
-Y_2\Gamma_3+Y_3\Gamma_2+Z\\
\ds -{1\over 2}Y_3\Gamma_1+{1\over 2}{\partial \psi_1\over \partial Y_2}\\
\ds {1\over 2}Y_2\Gamma_1+{1\over 2}{\partial \psi_1\over \partial Y_3}\\
e_{22}(\psi)\\  e_{23}(\psi)\\   e_{33}(\psi)
\end{pmatrix}\cdot
\begin{pmatrix}
-Y_2\Gamma_3+Y_3\Gamma_2+Z\\
\ds -{1\over 2}Y_3\Gamma_1+{1\over 2}{\partial \psi_1\over \partial Y_2}\\
\ds {1\over 2}Y_2\Gamma_1+{1\over 2}{\partial \psi_1\over \partial Y_3}\\
e_{22}(\psi)\\  e_{23}(\psi)\\   e_{33}(\psi)
\end{pmatrix}\\
\ge &  \int_\omega  
\GQ\,
 \begin{pmatrix}
-Y_2\Gamma_3+Y_3\Gamma_2\\
\ds -{1\over 2}Y_3\Gamma_1+{1\over 2}{\partial \overline{\chi}_1\over \partial Y_2}\\
\ds {1\over 2}Y_2\Gamma_1+{1\over 2}{\partial \overline{\chi}_1\over \partial Y_3}\\
e_{22}(\overline{\chi})\\  e_{23}(\overline{\chi})\\   e_{33}(\overline{\chi})
\end{pmatrix}\cdot
\begin{pmatrix}
-Y_2\Gamma_3+Y_3\Gamma_2\\
\ds -{1\over 2}Y_3\Gamma_1+{1\over 2}{\partial \overline{\chi}
_1\over \partial Y_2}\\
\ds {1\over 2}Y_2\Gamma_1+{1\over 2}{\partial \overline{\chi}
_1\over \partial Y_3}\\
e_{22}(\overline{\chi}
)\\  e_{23}(\overline{\chi}
)\\   e_{33}(\overline{\chi}
)
\end{pmatrix}
=\GA\begin{pmatrix}
\Gamma_1\\
\Gamma_2\\
\Gamma_3
\end{pmatrix}\cdot\begin{pmatrix}
\Gamma_1\\
\Gamma_2\\
\Gamma_3
\end{pmatrix}
\end{aligned}
\end{equation*} where
\begin{equation}\label{barchi}
\overline{\chi}=\Gamma_1\chi_1+\Gamma_2 \chi_2+\Gamma_3\chi_3\in \W.
\end{equation}
\end{lemma}
\begin{proof} First of all, it remains the same  to minimize the functional 
\begin{equation*}
(Z,\psi)\longmapsto \int_\omega  
\GQ\,
\begin{pmatrix}
-Y_2\Gamma_3+Y_3\Gamma_2+Z\\
\ds-{1\over 2}Y_3\Gamma_1+{1\over 2}{\partial \psi_1\over \partial Y_2}\\
\ds {1\over 2}Y_2\Gamma_1+{1\over 2}{\partial \psi_1\over \partial Y_3}\\
e_{22}(\psi)\\  e_{23}(\psi)\\   e_{33}(\psi)
\end{pmatrix}\cdot
\begin{pmatrix}
-Y_2\Gamma_3+Y_3\Gamma_2+Z\\
\ds -{1\over 2}Y_3\Gamma_1+{1\over 2}{\partial \psi_1\over \partial Y_2}\\
\ds {1\over 2}Y_2\Gamma_1+{1\over 2}{\partial \psi_1\over \partial Y_3}\\
e_{22}(\psi)\\  e_{23}(\psi)\\   e_{33}(\psi)
\end{pmatrix}
\end{equation*} over $\R\times H^1(\omega;\R^3)$ or $\R\times \W$. The second property of $\GQ$ in \eqref{propQ} implies  that the minimum is achieved for $Z=0$.  Then, it is easy to  prove that the minimum of the functional 
 over the space $\R\times \W$ is achieved for the element $(0,\overline{\chi})$, where $\overline{\chi}$ is given by \eqref{barchi}.
 
 \noindent  It is easy to   prove (e.g. by contradiction) that the norm 
$$||(\Gamma_1,\theta)||=\sqrt{\int_\omega\Big(\Big|-{1\over 2}Y_3\Gamma_1+{1\over 2}{\partial \theta
 \over \partial Y_2}\Big|^2+\Big|{1\over 2}Y_2\Gamma_1+{1\over 2}{\partial \theta
\over \partial Y_3}\Big|^2\Big)}$$ on the space $\R\times H^1(\omega)/\R$ is equivalent to the product norm of this space. Hence, there exists a positive constant $C$ such that for any $\Gamma_1\in \R$ and any function $\theta
\in H^1(\omega)$ we have
$$\int_\omega\Big(\Big|-{1\over 2}Y_3\Gamma_1+{1\over 2}{\partial \theta
 \over \partial Y_2}\Big|^2+\Big|{1\over 2}Y_2\Gamma_1+{1\over 2}{\partial \theta
\over \partial Y_3}\Big|^2\Big)\ge  C\big(\Gamma_1^2+||\nabla\theta
||^2_{L^2(\omega ; \R^2)}\big).$$ Taking into account the third property in \eqref{propQ}, this allows to prove the positivity of matrix $\GA$.
\end{proof}

\begin{remark}\label{10.7} In the case where the matrix $\GQ$ corresponds to an isotropic and homogeneous material, we have
\begin{equation*}
\GQ=\begin{pmatrix}
\lambda+2\mu & 0 & 0  & \lambda &0 & \lambda \\
0 & 2\mu & 0 & 0 & 0 & 0  \\
0 & 0 & 2\mu & 0 & 0 & 0  \\
\lambda & 0 & 0  & \lambda+2\mu & 0 & \lambda \\
0 & 0 & 0 & 0 & 2\mu & 0 \\
\lambda & 0 & 0 & \lambda & 0 &  \lambda+2\mu
\end{pmatrix}
\end{equation*} where $\lambda$, $\mu$ are the Lam's constants. We get
\begin{equation*}
\chi_1(Y_2,Y_3)=0,\qquad \chi_2(Y_2,Y_3)=\begin{pmatrix}
\ds -\nu Y_2Y_3 \\ \ds -\nu{Y^2_3-Y^2_2\over 2}
\end{pmatrix},\qquad \chi_3(Y_2,Y_3)=\begin{pmatrix}
\ds \nu{Y^2_2-Y^2_3\over 2}\\ \nu Y_2Y_3 
\end{pmatrix}
\end{equation*}  where $\nu$ is the Poisson's coefficient. The matrix $\GA$ is equal to
\begin{equation*}
\GA=\begin{pmatrix}
\ds {\pi\mu \over 4} & 0 & 0 \\
0 & \ds  {\pi E \over 4} & 0 \\
0 & 0 & \ds {\pi E \over 4}
\end{pmatrix}
\end{equation*} where $E$ is the Young's modulus.
\end{remark}
  
\end{document}